\documentclass[12pt]{article} 
\usepackage{mathtools}
\usepackage{boxedminipage}
\usepackage{amsfonts}
\usepackage{amsmath} 
\usepackage{mathabx}
\usepackage{amssymb}
\usepackage{graphicx}
\usepackage{amsthm}
\usepackage{t1enc}
\usepackage{subfig}
\usepackage[all,cmtip]{xy}
\usepackage{hyperref}

\newtheorem{theorem}{Theorem}[section]

\newtheorem{proposition}[theorem]{Proposition}
\newtheorem{corollary}[theorem]{Corollary}
\newtheorem{definition}[theorem]{Definition}
\newtheorem{example}[theorem]{Example}

\DeclareMathOperator{\Hom}{Hom}
\DeclareMathOperator{\End}{End}

\begin{document}

\title{Generalised homomorphisms, measuring coalgebras and extended symmetries}

\author{Marj Batchelor\thanks{Cambridge University,, Department of Pure Mathematics and Mathematical Statistics \href{mailto:mb139@cam.ac.uk}{mb139@cam.ac.uk}} \and Will Boulton \and Daren Chen\thanks{currently studying at Stanford University} \and Jonathan Rawlinson\thanks{Bristol University} \and Mustafa Warsi}

\maketitle

\begin{abstract}Three categories of algebras with morphisms generalising the usual set of algebra homomorphisms are described. The Sweedler product provides a hom-tensor equivalence relating these three categories, and a tool enabling the universal measuring coalgebra to be calculated in small cases. A parallel theory for modules is presented.  

\end{abstract}

In comparing algebras $A$ and $B$, while homomorphisms remain the standard morphisms between algebras, there are other linear maps - derivations, higher order derivations, twisted derivations and other non (co) commutative variants useful in non-commutative geometry - which record more subtle information than can be monitored through algebra homomorphisms. The theory of measuring coalgebras developed by Sweedler, in particular the universal measuring coalgebra $P(A,B)$, the terminal object in a category of measuring coalgebras, provides one framework for these more subtle comparisons. The principle drawback to working with the universal measuring coalgebra has been the absence of methods to describe and compute $P(A,B)$ even in simple cases. 

The search for a method by which the universal measuring coalgebra may be calculated has led to the  study of an algebra, the Sweedler product,  $A\triangleleft K$ which is a tensor between an algebra $A$ and a coalgebra $K$, creating a functorial action
\[
\triangleleft: Alg  \times Coalg \to Alg,
\]
where $Alg , \  Coalg$ are the (usual) categories of algebras and coalgebras respectively.  The result is simple: the universal measuring coalgebra $P(A,B)$ is the dual coalgebra  $(A\triangleleft B^\circ)^\circ$. 

This construction had also been developed by Joyal and Anel \cite{anel2013sweedler} for the differential graded setting, and many of their results are duplicated here - including the result stated above. This paper emphasises the computability of $A\triangleleft K$, and hence $P(A,B)$, explores examples, and further generalises the Sweedler product to categories of modules. We hope this paper will encourage those working in fields beyond categorical algebra to apply the concepts in their own fields. In that hope we include basic definitions, results and simple examples.

The paper is organised as follows. 

Section 1 introduces measuring coalgebras together with examples assuming no prior knowledge. It presents two bicategories of algebras, $\widecheck{Alg}$ and $\widehat{Alg}$ with morphisms in the first being categories of measuring coalgebras and in the second categories of extensions of algebras. These two categories offer related means of exploring the more subtle transformations of algebras. 

Section 2 reviews the basic facts about universal measuring coalgebras and the enriched category of algebras $\underline{Alg}$ whose hom-set $\underline{Alg}(A,B)$ is the universal measuring coalgebra $P(A,B)$. It includes the categorical framework for $A \triangleleft K$ and the principle results establishing natural correspondences 
\[
Alg(A\triangleleft K, B) \cong Alg(A, [K,B]) \cong Coalg(K, P(A,B)), 
\]
where $[K,B]$ denotes the convolution algebra of $K$ with $B$. 
The universal measuring algebra $F(A,B) = A\triangleleft B^\circ$ is introduced, and the properties which justify the name are summarised in Theorem \ref{thm:F(A,B)}.

Section 3 computes elementary examples of $F(A,B)$. Readers whose interest lies in potential applications may prefer to work through the examples before engaging with the categorical technicalities of section 2.

Section 4 follows the development of section 2 in the context of modules and comodules, arriving at a corresponding sequence of natural equivalences. While the interpretation and understanding of the corresponding objects in module categories is challenging, the beauty of the categorical approach is that the structure of the theory lies parallel with the more familiar theory in section 2. The module - comodule theory has one result without parallel in the algebra - coalgebra theory: the global category of modules enriched over the universal measuring comodule is fibred over a 2-category of modules with the category of measuring comodules as morphisms.

The final potential application in Section 4 relates to internal symmetry groups. Given an algebra $A$, (for example, the real algebra generated by $\sqrt {-1}$ and a conjugation operator $J$), the theory provides a universal algebra $F(A,A)$ and a canonical extension $A\to F(A,A)\otimes A$. Moreover, for any $A$ module $W$ there is a universal module $D(W,W)$ for $F(A,A)$ and an extension $W \to D(W,W)\otimes W$. If $W$ is a representation of a group $G$,  any $D(W,W)$ submodule $Z$ determines a group $G_Z$ of $\End(Z)$ generated by invertible elements in the image of $F(A,A)$. If $W$ is also a representation of a symmetry group $G$, any choice of $F(A,A)$ module $Z$ determines an associated symmetry group $G_Z$ and $G_Z \times G$ acts on $Z\otimes W$. 

About the authorship of this paper. This work was done over a period of several years with the help of summer research students, William Boulton, Jonathan Rawlinson, Daren Chen and Mustafa Warsi who were undergraduates at the time. All have moved on to work in other fields or other areas of mathematics.  William Boulton returned to the project to help with completion, and others have contributed advice and corrections. 

We would like to thank Queens' College Cambridge, Trinity College Cambridge and the Cambridge Summer Research in Mathematics programme for support for these students during their summer projects. In addition Marj Batchelor would like to thank Ed Miller for several very useful conversations in the early development of the ideas.  We are deeply indebted to Martin Hyland for his encouragement and guidance in categorical matters.

\section{Measuring coalgebras, notation and examples}

\subsection{Measuring coalgebras}
The goal of this section is to introduce measuring coalgebras and provide explicit examples. Here and throughout the paper vector spaces and algebras and their tensor products are over an arbitrary field $k$ unless otherwise specified.

A coalgebra $H$ is a vector space with a coassociative comultiplication.
\[
\triangle :H\rightarrow H\otimes H, \ \ \triangle h = \sum_{(h)} h_{(2)}\otimes h_{(1)}
\]
together with a counit $\epsilon: H\to k$ satisfying
\[
h = \sum_{(h)} h_{(2)}\epsilon(h_{(1)}) = \sum_{(h)} \epsilon(h_{(2)})h_{(1)}.
\]
where the sum $\sum_{(h)} h_{(2)}\otimes h_{(1)}$ is Sweedler notation identifying the equivalence class rather than any particular representative of the element in a tensor product. Coassociativity enables one to write powers of the comultiplication map
\[
\triangle^n:H \to  H\otimes ...\otimes H, \ \ \triangle^n (h) = \sum_{(h)} h_{(n+1)}\otimes ... \otimes h_{(1)}.
\]
\begin{definition}  Let $A$ and $B$ be algebras.  If $H$ is a coalgebra, a map
\[
\rho: H \rightarrow Hom_k(A,B)
\]
\emph{measures} if
\[
 <\rho(h),aa'> = \sum_{(h)} <\rho(h_{(2)}),a><\rho(h_{(1)}),a'>, \ \ \ <\rho(h),1_A> = \epsilon(h)1_B
\]
for $a, a'$ in $A$. 

Here, and throughout, < {\_}, {\_} > will denote evaluation. 
\end{definition}
The fundamental idea behind the paper springs from the following elementary observation.
\begin{proposition}
\label{prop:meas} Let $H$ be a coalgebra, and let $A, \ B$ be algebras as above.

\[
\rho: H \rightarrow Hom_k(A,B)
\]
is a measuring map if and only if the corresponding map
\[
\rho: A \rightarrow Hom_k(H,B)
\]
is an algebra homomorphism.
\end{proposition}
The proof is immediate given the convolution product: for $\alpha, \ \beta$ in $Hom_k(H,B)$, the product $\alpha * \beta$ is given by
\[
\alpha * \beta (h) = \sum_{(h)}<\alpha, h_{(2)}>< \beta, h_{(1)}>.
\]
The space $Hom_k(H, B)$, equipped with this product, is called the \emph{convolution product of $H$ with $B$}, and is denoted by $[H, B]$. 

\subsection{Examples of measuring coalgebras}\label{ssec:exmeasco} 
\begin{example} Homomorphisms.
\end{example}
Let $H$ be the coalgebra $kg$, where $\triangle g = g\otimes g$ and $\epsilon (g) = 1$. Observe that a linear map $\varphi: H \to \Hom_k(A,B)$ measures if and only if $\varphi(g)$ is an algebra homomorphism.

\begin{example}: Derivations. 
\end{example}
Let $H$ be the coalgebra with basis $ \{ g, \gamma \} $ with $\triangle g, \ \epsilon (g)$ as above, and 
\[
\triangle \gamma = \gamma \otimes g + g \otimes \gamma, \ \ \ \epsilon(\gamma) = 0.
\]
A linear map $\varphi: H \to \Hom_k(A,B)$ measures if and only if $\varphi(g)$ is a homomorphism and $\varphi(\gamma)$ is a derivation with respect to $\varphi(g)$.

Measuring coalgebras for a given pair $A, \ B$ of algebras forms a category  $\widecheck{Alg}(A,B)$ where maps between measuring coalgebras defined as follows.

\begin{definition} \label{def:measmap}  Let $H_1$ and $H_2$ be coalgebras with measuring maps $\rho_1, \ \rho_2: H_1, H_2 \to \Hom_k(A,B)$ respectively.  A \emph{map of measuring coalgebras} $\tau: H_1 \to H_2$ is a coalgebra map such that the following diagram commutes:
\[
\xymatrixcolsep{5pc} 
\xymatrix{
H_1 \ar[d]^{\tau} \ar[r]^{\rho_1} & \Hom_k(A,B) \\
H_2 \ar[ur]_{\rho_2} &
}
\] 
\end{definition}
The identity $1_H :H \to H$ provides the unit in $\widecheck{Alg}(A,B)$.

Dual coalgebras provides a rich collection of less familiar examples.

\begin{example} Extensions and dual coalgebras as measuring coalgebras.
\end{example}
Let $S$ be a finite dimensional algebra, and suppose we have an algebra homomorphism
\[
\sigma: A \to S \otimes B.
\]
Since $S$ is finite dimensional, the full linear dual $S^*$ is a coalgebra. The map corresponding map
\[
\sigma : S^* \to \Hom(A,B)
\]
is a measuring map. 
Thus extensions $\sigma: A \to S \otimes B$ for finite dimensional algebras $S$ are a source of examples of measuring coalgebras.

The finiteness restriction on $S$ can be dropped by introducing the dual coalgebra. 

\begin{definition} \label{def:dualcoalg} . Let $S$ be an algebra. The \emph{dual coalgebra} $S^\circ$ of an algebra $S$ is the subspace of the full linear dual $S^*$ of elements which vanish on a cofinite ideal of $S$. 
\end{definition}
See Sweedler \cite{Swee1969} Chapter VI.

Another familiar example is that of $\mathbb{C}[z]^\circ$, which has a basis $\{ \frac{d^n}{dz^n}|_b\}$ of elements vanishing on the ideal generated by $(z - b)^{n+1}$ with comultiplication familiar from the generalised product rule:
\[
\triangle \frac{d^n}{dz^n}|_b = \sum_k{ n\choose k} \frac{d^{n-k}}{dz^{n-k}}|_b\otimes \frac{d^k}{dz^k}|_b.
\]

The notation $\frac{d^n}{dx^n}|_b$ is extended to the case $n=0$, denoting evaluation at $b$. 

Extensions  $\sigma: A \to S \otimes B$ form a category in their own right, which we will denote by $\widehat{Alg}(A,B)$ in a manner parallel with  $\widecheck{Alg}(A,B)$. Define morphisms of extensions as follows.
\begin{definition}
\label{def:extmap}
 Let $S_1, \ S_2$ be algebras with homomorphisms $\sigma_i: A \to S_i \otimes B$. A homomorphism $\theta: S_1 \to S_2$ is a map of extensions if the following diagram commutes.
 \[
\xymatrixcolsep{5pc} 
\xymatrix{
A \ar[r] ^{\sigma_{1}} \ar[dr]_{\sigma_2} & S_1 \otimes B \ar[d] ^{\theta \otimes 1_B} \\
& S_2 \otimes B
}
\]
 \end{definition}
 Again the identity map $1_S:S \to S$ provides the unit in $\widehat{Alg}(A,B)$
 
\subsection{The bicategories  $\widecheck{Alg}$ and $\widehat{Alg}$ \label{sec:checkhat} }

The purpose of the paper is to establish a framework for comparing algebras in ways less restrictive than homomorphisms. Measuring coalgebras and extension offer two means of generalising the concept of homomorphism, offering new categories, but the increased generality needs additional structure enabling the comparison of these maps: each can be used to define new categories of algebras which carry the additional structure of a bicategory. 

\begin{proposition}
\label{prop:bicat}
\begin{enumerate}
\item There is a bicategory $\widecheck{Alg}$ whose objects are algebras, and whose set of morphisms $\widecheck{Alg}(A,B)$ is the category of measuring coalgebras $\rho: H \to \Hom_k(A,B)$.
\item There is a bicategory $\widehat{Alg}$ whose objects are algebras, and whose set of morphisms $\widehat{Alg}(A,B)$ is the category of extensions $\rho:A \to S\otimes B$.
\end{enumerate}
\end{proposition}
\begin{proof}
The case for $\widecheck{Alg}$.

The task is to show that morphisms - measuring coalgebras - "compose", that a suitable unit morphism exists, and that the composition of morphisms commutes suitably with the "vertical" morphisms of measuring coalgebras.

\emph{"Horizontal" morphisms $\odot$ - composition of measuring coalgebras.}

Suppose $H, \ K$ are measuring coalgebras with measuring maps $\phi: H \to \Hom(A,B), \ \psi: K \to \Hom(B,C)$ respectively. Then $K \otimes H$ is a coalgebra and the composition 
\[
\psi \otimes \phi : K \otimes  H \to \Hom(B,C)\otimes \Hom(A,B) \to \Hom (A,C)
\]
measures:
\begin{align}
<h\otimes k, aa'> = & \sum_{(h)}<h, <k_{(2)},a><k_{(1)},a'>> \\
	&= \sum _{(h),(k)} <h_{(2)},<k_{(2)},a>><h_{(1)},<k_{(1)},a'>> 
\end{align}
as required, since $\triangle (h \otimes k) = \sum_{(h),(k)} (h_{(2)} \otimes k_{(2)})\otimes (h_{(1)} \otimes k_{(1)})$

Define the composition of $H$ and $K$ in $\widecheck{Alg}$ to be the coalgebra $H\otimes K$ with the measuring coalgebra structure given by $\psi \otimes \phi$. This "horizontal" composition will be denoted by $H\odot K$. Composition of "vertical" morphisms - maps between measuring coalgebras will be denoted by $\circ$. 

\emph{The unit morphism in $\widecheck{Alg}$.}
Let E be the one dimensional coalgebra with the distinguished element $e$ with $\epsilon (e) = 1$, and $\triangle e = e \otimes e$.  For an algebra $A$ define $\eta_A:E \to \Hom(A,A)$ by $\eta (e) = 1_A$ giving $E$ the structure of a measuring coalgebra. Direct calculation confirms that the pair $E, \eta_A$ has the properties required of a unit completing the requirements that  $\widecheck{Alg}$ is a category.

\emph{Compatibility of "vertical" and "horizontal" maps}  Let $H_i, \ i = 1,\ 2, \ 3$ be in 
 $\widecheck{Alg}(A,B)$ and $K_i$ in  $\widecheck{Alg}(B,C)$, and let $\varphi_i :H_i \to H_{i+1}$, ($\psi_i :K_i \to K_{i+1}$) for $i = 1,\ 2$  be maps of measuring coalgebras. The requirement is to show that there are measuring maps
 \[
 \varphi_i\circ \psi_i : H_i \odot K_i \to  H_{i+1} \odot K_{i+1}, \ \  i = 1, \ 2
 \]
 and that these maps  satisfy
 \[
 (\psi_1\circ \psi_2)\odot (\varphi_1\circ \varphi_2) \cong (\psi_2\odot \varphi_2)\circ (\psi_1\odot\varphi_1)
 \]
The notation $\cong$ denotes a natural isomorphism guaranteed by the uniqueness of the tensor product rather than an equality.  That this compatibility requirement is fulfilled is follows from the fact that the horizontal composition $\odot$ is exactly the tensor product $\otimes$ for vector spaces, and the statement is true for tensor products of vector spaces.

The case for  $\widehat{Alg}$.

Again, the task is to define "horizontal" composition $\odot$ of extensions, verify that $\widehat{Alg}$ forms a category with objects being algebras, and morphisms being the categories $\widehat{Alg}(A,B)$, and confirm appropriate compatibility of "horizontal" and "vertical" morphisms. 

\emph{"Horizontal morphisms - composition of extensions"}

If $(S, \sigma)$ and $(T, \tau)$ are extensions from algebras $A$ to $B$ and $B$ to $C$ respectively, the composition $(T,\tau)\circ(S,\sigma)$ is defined via

\[
\xymatrixcolsep{5pc} 
\xymatrix{
A \ar[r] ^\rho& S\otimes B \ar[r] ^{1 \otimes \tau} & S \otimes (T\otimes C) \cong (S\otimes T) \otimes C
}
\]
making use of the natural isomorphism inherited from vector spaces.

The pair $(k, \mu)$ where $\mu: k \to k\otimes A = A$ is the unit in $A$ provides an identity map in the set of extensions from $A$ to $A$, so that $\widehat{Alg}(-,-)$ has the properties required of morphisms. 

\emph{Compatibility of "vertical" and "horizontal" maps.} 
The compatibility condition is the same as required for  $\widecheck{Alg}$. The condition is satisfied for exactly the same reason: the "horizontal" morphisms again are simply the tensor products of vector spaces.

\end{proof}

Let $\widecheck{Alg}(A,B)_f$ denote the full subcategory of $\widecheck{Alg}(A,B)$ whose objects are all finite dimensional measuring coalgebras.  Let $\widecheck{Alg}_f$ be the sub-bicategory whose morphisms are the restricted subcategories $\widecheck{Alg}(A,B)_f$. The full linear dual $H^*$ of a coalgebra $H$ is (always) an algebra. The content of proposition \ref{prop:meas} is precisely that in the case where $H$ is finite dimensional, the dual of a measuring $H \to \Hom_k(A,B)$ is precisely an extension $A \to H^* \otimes B$. Thus, proposition \ref{prop:meas} can be recast in categorical terms as follows.

\begin{theorem}
\label{thm:dual}
For $A$, $B$ algebras, taking the dual establishes an equivalence of categories
\[
*: \widecheck{Alg}(A,B)_f ^{op} \cong \widehat{Alg}(A,B)_f:*
\]
This establishes an equivalence of bicategories:
\[
*: \widecheck{Alg}_f ^{co} \cong  \widehat{Alg}_f:*.
\]
\end{theorem}

When the finiteness restriction is dropped the subtle differences between coalgebras and algebras becomes significant. The next section introduces universal measuring coalgebra $P(A,B)$ which is the terminal object in $\widecheck{Alg}(A,B)$. The search for a corresponding initial object $F(A,B)$ in $\widehat{Alg}(A,B)$ generated this work.

\section{Universal measuring coalgebras}\label{sec:univ} 

This section further develops the theory of measuring coalgebras through the introduction of the universal measuring coalgebra and the Sweedler product which can be combined to introduce an algebra, the universal measuring algebra, which serves as a pre-dual for the universal measuring coalgebra.

\subsection{The universal measuring coalgebra}

The contents of this subsection are covered in Sweedler for example. This section is included to provide sufficient background to enable those unfamiliar with coalgebras to use the results.

Coalgebras have a strong finiteness property which is fundamental for the constructions in this paper.  One consequence is that the universal measuring coalgebra $P(A,B)$, the terminal element in the category $\widecheck{Alg}(A,B)$, exists. 

\begin{proposition} \label{prop:Coalgebra finiteness} Given a coalgebra $C (\Delta, \epsilon)$ and $x \in C$, there exists a finite dimensional subcoalgebra $D \subset C$ such that $x \in D$.
\end{proposition}

See Sweedler page 46. Essentially, consider $\triangle^2 x = \sum x_{(3)}\otimes x_{(2)} \otimes x_ {(1)}$ and take $D$ to be the span of the elements $x_{(2)}$ in a minimal expression for $\triangle^2$. 

\begin{definition} \label{def:univ} The \emph{universal measuring coalgebra} $\pi :P(A,B) \to \Hom_k(A,B)$ is the unique measuring coalgebra with the property that there is a unique coalgebra map $P(\rho)$ such that the diagram
\[
\xymatrixcolsep{5pc}
\xymatrix{
P(A,B) \ar[r]^{\pi} & \Hom_k(A,B) \\
H \ar[u]^ {P(\rho)}\ar[ur]_{\rho} &
}
\]
commutes for any measuring coalgebra $\rho:H \to Hom_k(A,B)$.  
\end{definition}

The existence of the universal measuring coalgebra is guaranteed because the finiteness property ensures that the coproduct over all \emph{finite dimensional} coalgebras measuring $A$ to $B$ is sufficiently large to capture all measuring coalgebras. The fact that the category of coalgebras has both colimits and coequalisers, then allows $P(A,B)$ to be constructed as follows.

\begin{proposition} \label{prop:P(A,B)exists} $P(A,B)$ exists, and is a terminal object in $\widecheck{Alg}(A,B)$.
\end{proposition}

\begin{proof} See Sweedler \cite{Swee1969} p143, or CV \cite{hyland2017hopf} p18. $P(A,B)$ can be constructed as a coequaliser as follows.
For algebras $A$ and $B$, construct the coproduct.
 \[
 K = \coprod_{\kappa}H_{\kappa}
 \]
 where the coproduct is taken over all isomorphism classes of \emph{finite dimensional} coalgebras $(H_{\kappa}, \rho_{\kappa})$ that measure A to B. The finite dimensionality is critical for the coproduct to be well defined.  Define $L$ to be the coproduct
 \[
 L = \coprod_{\lambda: H_\kappa \to H_{\kappa '}} H_{\kappa, \lambda}
 \]
where $H_{\kappa, \lambda}$ is a copy of $H_\kappa$ for all $\lambda \in \widecheck{Alg}(A,B)$. Define two maps $f, g : L \to H$ by the image of each factor $H_{\kappa, \lambda}$. Define $f$ to be the map that sends $H_{\kappa, \lambda}$ identically to $H_\kappa$ followed by the inclusion of $H_\kappa$ in $K$.  Let $g$ be the map that sends $H_{\kappa, \lambda}$ to $H_{\kappa '}$ followed by the inclusion of $H_{\kappa '}$ in $K$. 

$P(A,B)$ to be the coequaliser of  this diagram; effectively
\[
P(A,B) = K/im (f-g).
\]

That the image of $f-g$ is a coideal of the coalgebra $K$, and that it is sent to 0 in $\Hom(A,B)$ need to be verified.  

That $P(A,B)$ is a terminal object in $\widecheck{Alg}(A,B)$ depends again upon the finiteness property of coalgebras: $P(A,B)$ is the ascending union of finite dimensional coalgebras. For an arbitrary measuring coalgebra $H$, the map of measuring coalgebras from $H$ to $P(A,B)$ can be built up as the union of the inclusions of finite dimensional subcoalgebras of $H$ in $K$. Being a terminal object in this category immediately shows uniqueness. 
\end{proof}

\begin{corollary} \label{cor:duco} 
For any algebra $S$, $S^\circ = P(S,k)$.
\end{corollary}
\begin{proof}
Evidently $S^\circ \to \Hom(S,k)$ measures so that there is a coalgebra map $S^\circ \to P(S,k)$.  This map is injective since composition with the measuring map $\pi$ is just the inclusion $S^\circ \to \Hom(S,k)$. Conversely, if $p$ is any element of $P(S,k)$, it is an element of a finite dimensional subcoalgebra $H$ of $P(S,k)$. The measuring map $\pi:H \to \Hom(S,k) = S^*$ corresponds to an algebra map $\pi:S \to H^*$, thus $\pi(p)$ vanishes on the cofinite kernel of that homomorphism, and hence $\pi(p)$ is in $S^\circ$.
\end{proof}

The correspondence between coalgebra maps $H \to P(A,B)$ and algebra homomorphisms $A \to [H,B]$ is made formal in the following corollary.

\begin{proposition} \label{prop:natequiv}   There is a natural equivalence
\[
Coalg(-, P(A,B))  \cong Alg(A, [-,B]).
\]
\end{proposition}
\begin{proof}
Again, see Sweedler, p 143. The symbol $\cong$ indicates that this is a natural equivalence, meaning that the correspondence between the two functors $Coalg(-, P(A,B))$ and $ Alg(A, [-,B])$ observed in \ref{prop:meas} commutes appropriately with coalgebra maps $\varphi: H \to K$: $Coalg(\varphi, P(A,B))$ corresponds to $Alg(A, [\varphi,B])$.
\end{proof}

\subsection{The Sweedler product}
\label{sec:square}

The principle tool in this paper, the Sweedler product, is an algebra, $A \triangleleft H$, a tensor of an algebra A with a coalgebra H which serves as an adjoint to $\Hom$ so that
\[
\Hom(A\triangleleft -, B) \cong \Hom(A, [-B])
\]
The idea was introduced in Anel and Joyal \cite{anel2013sweedler} for use in the differential graded setting, and the main theorems in this section are found in their paper on pages 122 to 124. This paper studies the construction as a practical tool for extending the scope of "morphism" in other settings, and generalises the construction to the setting of modules and comodules.
\begin{definition}

\label{def: square}Define $A\triangleleft H$, the \emph{Sweedler product} to be the quotient of the tensor algebra $T(A\otimes H)/J$ where $J$ is the minimal ideal which makes the following diagrams commute. 
\[
\xymatrixcolsep{3pc}
\xymatrix{
A \otimes A\otimes H \ar[d]^{\mu\otimes 1} \ar[r]^{1 \otimes \triangle}  &A\otimes A\otimes H \otimes H \ar[r]& (A\otimes H) \otimes (A\otimes H)\ar[r]& T(A\otimes H)  \ar[d]^s \\
A\otimes H \ar[rr]^r  & &T(A \otimes H) \ar[r]&A \triangleleft H & }
\]
\[
\xymatrixcolsep{3pc}
\xymatrix{
k\otimes H \ar[r]^{1_A\otimes 1} \ar[d]^{1\otimes \epsilon} & A\otimes H \ar[r] & T(A\otimes H) \ar[dr]  \\
k\otimes k \ar[r] &k \ar[r]& T(A\otimes H) \ar[r] & A\triangleleft H}
\]
\end{definition}
\begin{theorem} \label{thm:actions} There are functorial actions
\[
\triangleleft: Alg \times Coalg \to Alg
\]

\[
[-,-] :  Coalg^{op} \times Alg \to Alg
\]
providing a tensor hom adjunction
\[
Alg(A,[H,B]) \cong Alg(A \triangleleft H, B).
\]
\end{theorem}
\begin{proof}  
\emph{Functoriality}
The functoriality of the action $[-,-]$ is familiar.  For the construction of the Sweedler product, the outline of the argument is as follows.

Let $A_i$, $i = 1, \ 2 , \ 3$ be algebras and $H_i $ be coalgebras, let $f_i:A_i\to A_{i+1}, \ g_i: H_i \to H_{i+1}$, $i= 1, \  2$ be algebra and coalgebra homomorphisms respectively.  The first requirement is that the linear maps $f_i \otimes g_i:A_i \otimes A_{i+1} \otimes K_{i+1}$ should pass to algebra homomorphisms $f_i \triangleleft g_I : A_i\triangleleft K_i \to A_{i+1} \triangleleft K_{i+1}$, the second is that the composition should behave in a functorial manner:
\[
(f_2\circ f_1)\triangleleft (g_2\circ g_1) = (f_2 \triangleleft g_2)\circ(f_1 \triangleleft g_1)
\]

For $f_i \otimes g_i$ to pass to a homomorphism $f_i \triangleleft g_i$ it is necessary that the ideals $J_i$ defining the algebras $A_i\triangleleft H_i$ satisfy $T(f_i\otimes g_i) J_i \subset J_{i+1}$. Consider the diagrams

\[
\xymatrix{
A_i \otimes A_i \otimes H_i \ar[ddd]_f \ar[dr] ^m \ar[rrr]^a& & & A_i \otimes  H_i \otimes A_i \otimes H_i  \ar[dl]_c \ar[ddd]^f \\
&A_i \otimes H_i \ar[d]^f \ar[r]^j & T(A_i \otimes H_i )\ar[d]_k& \\
&A_{i+1} \otimes H_{i+1}\ar[r] ^j& T(A_{i+1} \otimes H_{i+1}) & \\
A_{i+1} \otimes A_{i+1} \otimes H_{i+1}\ar[rrr] ^a\ar[ur]^m & & & A_{i+1} \otimes H_{i+1}\otimes A_{i+1} \otimes H_{i+1} \ar[ul]_c
}
\]

Here maps labelled $a$ have to do with comultiplication on $H$, maps labelled $m$ involve multiplication in $A$, maps labelled $f$ have to do with maps $A_i \to A_{i+1}$ and $H_i \to H_{i+1}$, are the inclusion of the appropriate subspaces in the tensor algebra, and $k$ is the map induced by $f_i\otimes g_i$ on the tensor algebra.

\[
\xymatrix{
k \otimes H_i \ar[ddd] _f\ar[dr] ^e\ar[rrr]^u& & &A_i\otimes H_i \ar[dl]_a \ar[ddd]^f\\
&k \otimes k\ar[d]_{Id} \ar[r]^b & T(A_i \otimes H_i )\ar[d]^k & \\
&k\otimes k \ar[r]^b & T(A_{i+1} \otimes H_{i+1}) & \\
k\otimes H_{i+1}\ar[rrr] ^u\ar[ur] ^e& & & A_{I+1} \otimes H_{i+1} \ar[ul]_g
}
\]

Here the maps labelled $u$ are the inclusions of $k$ as units in the $A_i$, the maps $e$ have to do with the counits in the $A_i$, $b$ is the inclusion of $k\otimes k \to k$ in the tensor algebra. $k$ again is the map on the tensor algebra induced by $f_i\otimes g_i$.

The requirement is that $k(J_i) \subset J_{i+1}$. Observe that the ideal $J_i$ is generated by the images of $j\circ m - c \circ a$ in the first diagram and $a \circ u - e \circ f$ in the second.  The commutativity of the two diagrams ensures that the generators of $J_i$ are sent by $k$ into the generators of $J_{i+1}$, hence  $k(J_i) \subset J_{i+1}$.

Functoriality follows as a consequence of the fact that the statement is true for the respective tensor algebras: 
\[
T(f_2\circ f_1 \otimes g_2 \circ g_1) = T(f_2 \otimes g_2) \circ T(f_1 \otimes g_1).
\]

To see that $A\triangleleft H$ has the desired adjunction properties, observe that the following three sets are naturally identified as vector spaces:
\[
\Hom_k(A,[H,B]) \cong \Hom_k(A\otimes H,B) \cong \Hom_k(H, \Hom_k(A,B)).
\]

An element $\varphi$ in $\Hom_k(A,[C,B])$ is an algebra homomorphism exactly when the corresponding element in $\Hom_k(C, \Hom_k(A,B))$ is a measuring map. But the condition that $\varphi$ measures is exactly the statement that the algebra map 
\[
T\phi:T(A\otimes H) \to B
\]
factors through $A\triangleleft H$. As with proposition \ref{prop:natequiv}
naturality is a consequence of the underlying fact for vector spaces.
\end{proof}

\subsection{Measuring coalgebras for coalgebras \label{sec:meas.co.for.co}}

In the previous section, the bicategories $\widecheck{Alg}$ and $\widehat{Alg}$ were introduced to provide a context which includes more general morphisms between algebras. The existence of  the universal measuring coalgebra provides another approach, that of enriched categories, to be introduced in the next section, replacing the morphism set $Alg(A,B)$ with the measuring coalgebra $P(A,B)$. 

Coalgebras carry a similar structure. There is a parallel concept of measuring coalgebra for coalgebras which introduces the possibility of generalised maps between coalgebras, and a final object in the category of measuring coalgebras for coalgebras, which will enable us to establish a full equivalence between these three approaches to the generalisation of homomorphisms.

\begin{definition} Let $H, \ K, \ L$ be coalgebras. A linear map $\lambda: L \to \Hom_k(H, K)$ is said to \emph{measure} if the map $\lambda$ regarded as an element of $\Hom_k(L \otimes H, K)$ is a coalgebra map. As with measuring coalgebras for algebras, maps between measuring coalgebras $L_1, \ L_2$ for coalgebras are those for which the diagram 
\[
\xymatrixcolsep{5pc}
\xymatrix{
L_1 \ar[d] \ar[r]^{\lambda_1} & \Hom_k(H,K) \\
L_2 \ar[ur]_{\lambda_2} &
}
\]
commutes.
\end{definition}

\begin{definition}
For coalgebras $H, \ K$, the \emph{universal measuring coalgebra}  $P(H,K)$ is a coalgebra with a map  $\pi:P(H,K) \to \Hom_k(H,K)$ with the property that given any measuring coalgebra $\lambda: L\to \Hom_k(H,K)$ there is a unique map of coalgebras making the following diagram commute.
\[
\xymatrixcolsep{5pc}
\xymatrix{
L \ar[d] \ar[r]^{\lambda} & \Hom_k(H,K) \\
P(H,K) \ar[ur]_{\pi} &
}
\]
\end{definition}

As for algebras, there is no obstruction to forming the required coequalisers, so that there is a terminal object in the category of measuring coalgebras for any pair of coalgebras $(H,K)$; i.e. $ P(H,K)$ exists.

Proposition \ref{prop:natequiv} for the universal measuring coalgebra for algebras has its parallel for the universal measuring coalgebra for coalgebras.
\begin{proposition} \label{prop:natequiv.co}There is a natural equivalence
\[
Coalg(-, P(H,K)) \cong Coalg(-\otimes H, K)
\]
\end{proposition}

\subsection{Enriched categories of algebras and coalgebras\label{sec:enriched}}

The concept of enriched categories, in which morphisms and compositions of morphisms may be objects and morphisms in a category more general than sets,  provides an alternative to the bicategory $\widecheck{Alg}$ as a means of including maps between algebras respecting the multiplicative structure in more sensitive ways.  
\begin{proposition}
\begin{enumerate}
	\item There is a category $\underline{Alg}$ whose objects are algebras and whose hom-sets are the universal measuring coalgebras $P(A,B)$. $\underline{Alg}$ has the structure of an enriched category.
	\item There is a category $\underline{Coalg}$ whose objects are coalgebras and whose hom-sets are the universal measuring coalgebras $P(H,K)$. $P(H,K)$ has the structure of an enriched category.
\end{enumerate}
\end{proposition}

\begin{proof}
The universal property of the universal measuring coalgebra guarantees the composition of morphisms.  The uniqueness of the universal measuring map guarantees associativity of composition and also guarantees that the identity morphism satisfies its requirements.
\end{proof}

Proposition \ref{prop:natequiv} extends to the enriched setting as follows.

\begin{theorem} \label{thm:natequiven} There are natural equivalences
\[
\underline{Alg}(A\triangleleft -, B) \cong \underline{Alg}(A,[-,B]) \cong \underline{Coalg}(-, \underline{Alg}(A,B))
\]
\end{theorem}

\begin{proof}

Let $L$ be a coalgebra. The aim is to establish the existence of natural isomorphisms
\[
P(A\triangleleft L, B) \cong P(A, [L,B]) \cong P(L, P(A,B))
\]

By Yoneda's lemma, it is sufficient to show that for any coalgebra $C$ there are natural isomorphisms of sets
\[
Coalg(C,P(A\triangleleft L, B)) \cong Coalg(C,P(A, [L,B])) 
\]
\[
Coalg(C,P(A, [L,B]))  \cong Coalg(C, P(L,P(A,B))) 
\]

The work has either already been done or is familiar from vector spaces.  For the first equivalence:
\begin{eqnarray*}
Coalg(C, P(A\triangleleft L, B)) &  \cong ^1 & Alg (A\triangleleft L, [C,B]) \\
				&\cong^2  & Alg(A, [L,[C,B]]) \\
				& \cong ^3 & Alg(A, [L\otimes C, B]) \\
				& \cong ^4& Alg( A, [C\otimes L, B]) \\
				& \cong^5 & Alg(A, [C,[L,B]]) \\
				& \cong^6 & Coalg (C, P(A, [L,B]).			
\end{eqnarray*}
Here, equivalence 1 is the consequence of theorem \ref{prop:meas} and the universal property of universal measuring coalgebras. Equivalence 2 is the adjunction of theorem \ref{thm:actions}. The final equivalence again follows from proposition \ref{prop:meas} and the property of the universal measuring coalgebra.
For the second:
\begin{eqnarray*}
Coalg(C, P(L, P(A,B))) & \cong^7 & Coalg(C \otimes L, P(A,B)) \\
				& \cong^8 & Alg(A, [C \otimes L, B])\\
				& \cong ^9& Alg (A, [C,[L,B]]) \\
				& \cong^{10} & Coalg( C, P(A, [L,B])).
\end{eqnarray*}
Here, equivalence 7 is proposition \ref{prop:natequiv.co}, the universal property for the universal measuring coalgebra for coalgebras. Equivalence 8 makes use of proposition \ref{prop:natequiv}, as does 10.
\end{proof}

\subsection{ The universal measuring algebra $F(A,B) := A\triangleleft B^\circ $ and the calculation of $P(A,B)$ \label{sec:F(A,B)}}

The motivation behind our development of the Sweedler algebra was the observation that it provided a pre-dual which enabled explicit computation of $P(A,B)$ in the case where $B$ is finite dimensional. The fundamental observation is the following specific case of theorem \ref{thm:natequiven}

\begin{corollary}\label{cor:F(A,B)*}
$P(A,[B^\circ, k]) \cong P(A\triangleleft B^\circ, k) \cong (A\triangleleft B^\circ)^\circ.$

In the case where $B$ is finite dimensional, $P(A,B) \cong P(A\triangleleft B^\circ, k).$

\end{corollary}

\begin{definition} \label{def:F(A,B)} Define the \emph{universal measuring algebra} $F(A,B)$ to be the Sweedler product

\[
F(A,B):= A\triangleleft B^\circ .
\] 

\end{definition}

The name is chosen to suggest its role linking the algebra structures of $A$ and $B$. The properties of $F(A,B)$ are summarised in the following theorem.

\begin{theorem} \label{thm:F(A,B)} The properties of $F(A,B)$ are as follows.
\begin{enumerate}
	\item For any finite dimensional algebra $B$ and $A$ any algebra there is an algebra homomorphism
	\[
	\eta(A,B): A \to F(A,B)\otimes B 
	\] 
	such that for any extension $\sigma: A \to S\otimes B$ there is a unique homomorphism $F(\sigma):F(A,B) \to S$ such that the diagram 
\[	
\xymatrixcolsep{5pc}
\xymatrix{
A\ar[rd]_\sigma \ar[r]^{\eta(A,B)} & F(A,B)\otimes B \ar[d]^{F(\sigma)\otimes 1_B}\\
&S\otimes B
}
\]
commutes.  Moreover, $\eta(-,B)$ is a natural transformation.

	\item For arbitrary B the algebra homomorphism $F(\sigma):F(A,B) \to S$ still exists: the assignment
	\[
	F: \widehat{Alg}(A,B) \to Alg(F(A,B), -), \ \ \ F(S,\sigma) = F(\sigma):F(A,B) \to S
	\]
	is a functor.
	\item If $B$ is finite dimensional, $F(A,B)$ is an initial object in $\widehat{Alg}(A,B)$.
	\item For $A$ finite dimensional, $F(A,A)$ is a bialgebra. More generally for finite dimensional algebras $A, \ B, \ C$ there is a factorisation
	\[
	\triangle_B:F(A,C) \to F(A,B)\otimes F(B,C) 
	\]
	which satisfies a coassociativity identity:
	\[
	(1_{F(AB)}\otimes \triangle_C )\circ \triangle_B = (\triangle_B \otimes 1_{F(C,D)} )\circ \triangle_C: F(A,D) \to F(A,B)\otimes F(B,C)\otimes F(C,D).
	\]
	\item The evaluation map
	\[
	A \to [P(A,B), B]
	\]
	is an algebra homomorphism.
	\item $P(A,B) \cong F(A,B)^\circ$
\end{enumerate}

\end{theorem}
\begin{proof}
\begin{enumerate}
\item Define $\eta(A,B):A \to F(A,B)\otimes B$ using the identity element $e$ in $B^\circ \otimes B$ so that $\eta_A(a)$  is the image of $a \otimes e$ in $A\otimes B^\circ \otimes B \subset T(A \otimes B^\circ) \otimes B $ in the quotient $F(A,B)\otimes B$. 

Check that $\eta(A,B): A \to F(A,B)\otimes B$ is an algebra homomorphism. 
Consider the diagram:
\[
\xymatrixcolsep{2pc}
\xymatrix{
A\otimes A\ar[r] \ar[d] & A\otimes A \otimes B^\circ \otimes B \otimes B^\circ \otimes B \ar[r] & A \otimes B^\circ \otimes B \otimes A \otimes B^\circ \otimes B \ar[d]\\
A\otimes A \otimes B^\circ \otimes B \ar[d] \ar[r] & A\otimes A\otimes B^\circ \otimes B^\circ \otimes B \ar[r] & (A\otimes B^\circ ) \otimes (A\otimes B^\circ) \otimes B \ar[d]\\
 A \otimes B^\circ \otimes B \ar[r]& F(A,B) \otimes B & F(A,B) \otimes B \otimes F(A,B)\otimes B \ar[l]
}
\]
Going anticlockwise from top left, the diagram represents the product of elements in $A$, followed by their inclusion into $F(A, B) \otimes B$. Going clockwise, the morphisms map $A$ to $F(A, B) \otimes B$ by $\eta$ first, and then multiply elements (first multiplying elements in $B$, then in $F(A, B)$). Thus the perimeter arrows are the required identity to verify that $\eta(A,B): A \to F(A,B) \otimes B$ is an algebra homomorphism.

To see that this diagram commutes; the bottom square is just the diagram defining the quotient ideal of $F(A, B)$, tensored with $B$. For the top square, consider the diagram relating multiplication in $B$ with comultiplication in $B^\circ$: 
\[
\xymatrixcolsep{5pc}
\xymatrix{
B\otimes B \otimes B^\circ \ar[r] \ar[d] & B\otimes B\otimes B^\circ \otimes B^\circ \ar[r] & B\otimes B^\circ \otimes B \otimes B^\circ \ar[d] \\
B \otimes B^\circ \ar[r] & k & k \otimes k \ar[l]
}
\]
Since $B$ is finite dimensional the dual algebra of $B^\circ $ is just $B$, this diagram dualises with reversed arrows; tensored with $A\otimes A$ this is essentially the top square above - other than some shuffling, the factors of $A$ are otherwise untouched. 

The statement that $\eta(-,B) $ is natural means that If $A'$ is another algebra with a homomorphism $\tau:A \to A'$ the diagram
\[	
\xymatrixcolsep{5pc}
\xymatrix{
A\ar[d]_\tau \ar[r]^{\eta(A,B)} & F(A,B)\otimes B \ar[d]^{F(\tau, 1_B)}\\
A' \ar[r]^{\eta(A'B)} & F(A',B)\otimes B
}
\]
commutes. (A more correct statement would be that $\eta: \mathbb{Id} \to F(-,B)\otimes B$ is a natural transformation.) The proof is a matter of checking that the map
\[
\tau \otimes e: A \otimes B^\circ \otimes B  \to  A' \otimes B^\circ \otimes B 
\]
induces a homomorphism 
\[
\tau \otimes e:T( A \otimes B^\circ )\otimes B  \to  T(A' \otimes B^\circ) \otimes B 
\]
that respects the ideal defining $F(A,B)$.
The dependence on the second variable, $\eta(A, -)$ is not so simple.

\item A homomorphism $\sigma : A \to S \otimes B$ can be regarded as an algebra homomorphism $\sigma: A \to [B^\circ, S]$ by mapping $s \otimes b \mapsto s\otimes \psi_b$; $\psi_b$ denoting evaluation at $b$.
The equivalence in \ref{thm:actions} $Alg(A, [B^\circ, S]) \cong Alg (A\triangleleft B^\circ,S)$  provides the required map $A\triangleleft B^\circ = F(A,B) \to S$.

The category $Alg(F(A,B),-)$ is the category with pairs $(\rho, T)$ where $T$ is an algebra and $\rho:F(A,B) \to T$ is an algebra homomorphism. Morphisms $\tau: (\rho, T) \to  (\rho', T')$ are homomorphisms $\tau: T \to T'$ such that $\rho' = \tau \circ \rho$. Functoriality follows from the functoriality of $\triangleleft$, theorem \ref{thm:actions}.

\item From part 1 above, $\eta(A,B) : A \to F(A,B) \otimes B$ is an extension. Check that the map $\sigma$ of part 2 above is a morphism of extensions. 

\item This follows from uniqueness of $\eta$ For algebras $A, \ B, C$ consider the composition
\[
\xymatrixcolsep{5pc}
\xymatrix{
A\ar[r]^{\eta(A,B)} & F(A,B) \otimes B \ar[r]^{1\otimes \eta(B,C)} & F(A,B) \otimes F(B,C) \otimes C.
}
\]
This provides a factorisation $\triangle_B :F(A,C) \to F(A,B) \otimes F(B,C)$.  The coassociativity follows from the universal property of $\eta$. In the case $A=B=C$, this provides F(A,\ A) with a coassociative comultiplication. $F(A, \ A)$ has a counit determined on generators by the contraction $A\otimes A^\circ \to k$.
\item This is \ref{prop:meas}.
\item This is a restatement of proposition \ref{cor:F(A,B)*}.
\end{enumerate}
\end{proof}

Some elementary examples of $F(A,B)$ are collected in the following proposition.
\begin{proposition}
\begin{enumerate}
	\item $F(A,k) = A$
	\item $F(k,A) = k$
\end{enumerate}
\end{proposition}
\emph{Remark} Part 2 of \ref{thm:F(A,B)} justifies the naming of this construction as the universal measuring algebra. For $A = B$  it seems to play a role reminiscent of universal enveloping algebras, providing an object that maps to all representations of $A$.

\begin{corollary}
If $\theta:A \to \End(W)$ is a representation of $A$, $F(\theta)$ is an algebra homomorphism $F(\theta):F(A,A) \to End(W)$ 
\end{corollary}
\begin{proof}
This is immediate from part 2 of \ref{thm:F(A,B)} above, as $\theta$ can be considered as the extension 
\[
A \to  \End(V)\otimes A, \ \ \ a\mapsto \theta(a)\otimes 1.
\]
\end{proof}

\section{Computations}
The categorical structure behind the constructions have the consequence that they are easy to work with.  They are also computable in small cases.  This section begins with the basic example that initiated the project, the computation of the universal measuring coalgebra $P(\mathbb{C}, \mathbb{C})$ as $F(\mathbb{C}, \mathbb{C})^ \circ$, where $\mathbb{C}$ is considered as an algebra over $\mathbb{R}$. 
A second set of computable examples are the algebras $F(\mathbb{Q}[x]/p(x),\mathbb{Q}[x]/p(x))$ for a polynomial $p(x)$.
Finally for the particular case of $p(x) = x^2$, $F(\mathbb{R}[x]/x^2,\mathbb{R}[x]/x^2)$ coincides with the Pareigis Hopf algebra, connecting comodules over this coalgebra with chain complexes. 

\subsection{The algebra  $F(\mathbb{C}, \mathbb{C})$} \label{sec;basicexample}
Throughout this subsection let $A = \mathbb{R} [x]/(x^2+1)$, and let $F$ be the (real) algebra $F(A,A)$.  
The generators of the ideal defining $F(A,A)$ can be got directly from the definitions, but it is easier to compute postulate an algebra $F$ together with a homomorphism $\eta: A \to F\otimes A$ with the required universal property.

Write $\eta (x) = f_0 \otimes 1 + f_1 \otimes x$. Since $\eta$ is assumed to be an algebra homomorphism, it must be the case that $\eta(x)^2 = -1$ or
\begin{align}
f_0^2 - f_1^2 & =  -1 \\
f_0f_1 +f_1 f_0 & = 0
\end{align}

Let $F$ be the algebra over $\mathbb{R}$ generated by $\{f_0, \ f_1\}$ subject to the above identities. If $S$ is another algebra with a homomorphism $\sigma: A \to S \otimes A$, write $\sigma(x)$ as $\sigma(x) = s_0 \otimes 1 + s_1 \otimes x$ with $s_i$ in $S$.. Evidently $s_0, \ s_1$ must satisfy the same identities, and hence there is an algebra homomorphism $F(\sigma): F \to S$, satisfying the universal property required of $F(\mathbb{C},\mathbb{C})$.

$F(\mathbb{C},\mathbb{C})$ can then be given explicitly in terms of a basis:
\[
F = < f_0^\epsilon f_1^k : \epsilon \in \{0,1\}, k \in \mathbb{Z}^+ > .
\]

\subsubsection{The coalgebra structure on  $F(\mathbb{C}, \mathbb{C})$}
The coalgebra structure on $F$ can be given by considering the composition 
\[
\xymatrixcolsep{5pc}
\xymatrix{
A \ar[r]^\eta &F\otimes A\ar@/^/[r]^{1\otimes \eta}\ar@/_/[r]_{\triangle \otimes 1} & F \otimes F\otimes A  
}
\]
The comultiplication $\triangle$ is defined so that both compositions are equal.
\begin{align}
(1\otimes \eta) \circ \eta(x) &= (1\otimes \eta) (f_0\otimes 1 + f_1 \otimes x) \\
	& = f_0 \otimes 1 \otimes 1 + f_0 \otimes f_0 \otimes 1 + f_1 \otimes f_1, \\
(\triangle \otimes 1)\circ \eta & = \triangle f_0 \otimes 1 + \triangle f_1 \otimes x	
\end{align}
By comparing coefficients of $A$ 
\[
\triangle f_1 = f_1 \otimes f_1, \ \ \triangle f_0 = f_0 \otimes 1 + f_1 \otimes f_0.
\]
	
\subsubsection{Finite dimensional modules of $F$ and $P(\mathbb{C},\mathbb{C})$}


The quest that generated the development of the theory contained in this paper was the quest to calculate $P(\mathbb{C},\mathbb{C})$ considering $\mathbb{C}$ as an $\mathbb{R}$ algebra.  
\begin{proposition}
Irreducible subcoalgebras of $P(\mathbb{C},\mathbb{C})$ are of the form $J_b^\perp$
where $b \in \mathbb{C}$ and $J_b$ is the ideal of $F$ generated by $f_1^2 - b^2$ in the case that $b^2$ is real, or $(f_1 - b)(f_1 -\bar{b})$ otherwise, and $J_b^\perp$ is the subcoalgebra of $F$ which vanishes on $J_b$.
\end{proposition}
\begin{proof}
To calculate $P(\mathbb{C}, \mathbb{C}) = F^\circ$, the strategy is to explore finite dimensional complex representations $\theta:F\to \End(V)$.  The real dual coalgebra  $F^\circ$ is then the union of the real dual coalgebras $\theta^*(\End(V)^*)$. 

Suppose $V$ is such a finite dimensional complex representation of $F$. Let $V_b$ be an $f_1$ eigenspace with eigenvalue $b$, and suppose that $v \in V_b$ is an eigenvector.  Then $f_0v$ is also an eigenvector of $f_1$ with eigenvalue $-b$. The subspace spanned by $\{v, f_0v\}$ is then an irreducible $F$ submodule, since $f_0^2v = (f_1^2 -1)v$.

Thus finite dimensional irreducible modules of $\mathbb{C}F$ are indexed by pairs of  complex numbers $ \{b, a\}$ with $a^2 - b^2 + 1 = 0$. The image of $F$ is isomorphic to $M_2(\mathbb{C})$. The kernel $J_b$ of $\mathbb{C}F \to \End (V) $ is generated by $(f_1^2 - b^2)$. If $b^2$ is real, this generates a maximal ideal of $F$. If $b^2$ is not real, $(f_1 -b)(f_1- \bar{b})$ is in $F$ and generates an ideal. 
\end{proof}

\subsection{Representations of $F$ in $M_2(\mathbb{R}(\lambda))$} 
\label{sec:Freps}
The algebra $F$ can be represented as a subalgebra of  $M_2(\mathbb{R}(\lambda))$ by assigning
\[
f_1 \to \begin{bmatrix} a(\lambda) & 0 \\ 0 & -a(\lambda) \end{bmatrix}, \ \  \ \ \ 
f_0 = \begin{bmatrix} 0 & b(\lambda) \\ c(\lambda) & 0 \end{bmatrix}
\]
where $a(\lambda), \ b(\lambda), \ c(\lambda)$ are such that $b(\lambda)c(\lambda) - a(\lambda)^2 + 1 = 0$. The dependence on an indeterminate $\lambda$ is needed to ensure that powers of $f_0$ and $f_1$ are algebraically independent from one another (excluding the two defining relations $f_0^2 - f_1^2 + 1 = f_0f_1 + f_1f_0 = 0$). Examples suggesting a interpretation as deformations of complex numbers are
\[
f_1 \to \begin{bmatrix} \lambda & 0 \\ 0 & -\lambda \end{bmatrix}, \ \  \ \ \ 
f_0 = \begin{bmatrix} 0 & \lambda+1 \\ \lambda-1 & 0 \end{bmatrix}
\]
and 
\[
f_1 \to \begin{bmatrix} \sin(\lambda) & 0 \\ 0 & -\sin(\lambda) \end{bmatrix}, \ \  \ \ \ 
f_0 = \begin{bmatrix} 0 & \cos(\lambda) \\ -\cos(\lambda) & 0 \end{bmatrix}
\]
(allowing more general functions of $\lambda$). These can also be viewed as curves in $\mathfrak{sl}_2(\mathbb{R})$. 

Such representations of $F$ enables us to give a more concrete description of $F^\circ$. Since $F\subset M_2(\mathbb{R}[\lambda])$, there is a coalgebra map 
\[
M_2(\mathbb{R}[\lambda])^\circ= \mathbb{R}[\lambda]^\circ \otimes M_2(\mathbb{R})^\circ \to F^\circ
\]
The coalgebras $\mathbb{R}[\lambda]^\circ $ and $M_2(\mathbb{R})^\circ \to F^\circ$ were described in Section \ref{ssec:exmeasco}, although we need to note the slight subtlety that in section \ref{ssec:exmeasco} we considered $\mathbb{C}[\lambda]^\circ $ rather than $\mathbb{R}[\lambda]^\circ $. 
The necessary modification for the real case is that $\mathbb{R}[\lambda]^\circ $ has a real basis $\{\frak{Re} \frac{d^n}{d\lambda^n}|_b, \frak{Im} \frac{d^n}{d\lambda^n}|_b\} $ for complex $b$. 
 Thus $M_2(\mathbb{R}[\lambda])^\circ$ has a basis $\{ \frak{Re}\frac{d^n}{d\lambda^n}|_b \xi_{ij}, \frak{Im}\frac{d^n}{d\lambda^n}|_b \xi_{ij} \}$.
 
 $F$ is a subalgebra of $M_2(\mathbb{R}[\lambda])$, so the map  $M_2(\mathbb{R}[\lambda])^\circ \to F^\circ$ has a kernel. A general element $M(\lambda)$ in $F$ can be written as
\[
 M(\lambda) = \begin{bmatrix} p(\lambda) & q(\lambda)(\lambda + 1)\\
 		 -q(-\lambda)(-\lambda +1) & p(-\lambda) \end{bmatrix},
\]
 for poilynomials $p(\lambda), \ q(\lambda)$. Thus 
 \begin{align}
 \frac{d^n}{d\lambda^n}|_b<\xi_{22},M(\lambda)> =  &(-1)^n \frac{d^n}{d\lambda^n}|_{-b}<\xi_{22},M(\lambda)>\\
  \frac{d^n}{d\lambda^n}|_b<\xi_{12} , M(\lambda)> =  & (-1)^{n+1} \frac{d^n}{d\lambda^n}|_{-b}<\xi_{21},M(\lambda)>.
 \end{align}
 Thus $F^\circ$ has a basis  $\{ \frak{Re}\frac{d^n}{d\lambda^n}|_b \xi_{11}, \frak{Im}\frac{d^n}{d\lambda^n}|_b \xi_{11} , \frak{Re}\frac{d^n}{d\lambda^n}|_b \xi_{12}, \frak{Im}\frac{d^n}{d\lambda^n}|_b \xi_{12}\}$.

 and comultiplication given by
\begin{align}
\triangle \frak{Re}\frac{d^n}{d\lambda^n}|_b \xi_{ij} = & \sum_{m,k} \binom{n}{m} \frak{Re}\frac{d^m}{d\lambda^m}|_b \xi_{ik}\otimes  \frak{Re}\frac{d^{n-m}}{d\lambda^{n-m}}|_b \xi_{kj} -  \sum_{m,k} \binom{n}{m} \frak{Im}\frac{d^m}{d\lambda^m}|_b \xi_{ik}\otimes  \frak{Im}\frac{d^{n-m}}{d\lambda^{n-m}}|_b \xi_{kj}  \\
\triangle \frak{Im}\frac{d^n}{d\lambda^n}|_b \xi_{ij} = & \sum_{m,k} \binom{n}{m} \frak{Re}\frac{d^m}{d\lambda^m}|_b \xi_{ik}\otimes  \frak{Im}\frac{d^{n-m}}{d\lambda^{n-m}}|_b \xi_{kj} +  \sum_{m,k} \binom{n}{m} \frak{Im}\frac{d^m}{d\lambda^m}|_b \xi_{ik}\otimes  \frak{Re}\frac{d^{n-m}}{d\lambda^{n-m}}|_b \xi_{kj}  
\end{align}
This presentation of $F$ and $F^\circ$ provides a hands-on approach to using $F$.  
The finite dimensional subcoalgebras 
\[
F^\circ_b = < \frac{d^n}{d\lambda^n}|_b \xi_{ij}| n\geq 0, i, \ j, \in \{1,2\}>
\]
for $n \leq N$ provide finite dimensional representations of $F$ which could be induced to provide representations of 
$M_2(\mathbb{R}[\lambda, \lambda^{-1}])$ and related algebras.  

\subsection{Matrix methods}
\label{ssec:matcomp}
The method used to calculate $F(\mathbb{C},\mathbb{C})$ extends to cases where $A$ and $B$ are finite dimensional with dim$B = n$ using matrices. Choose a basis $\{ a_i\}$ (resp. $\{b_r\}$) for $A$ (resp. $B$) and let $B_r$ be the matrices representing $b_r$ in the left regular representation of $B$. Let $\{ b_r^*\}$ denote the dual basis.  Let  $\{c_{ij}\}$ and $\{d_{rs}\}$ be the structure constants 
\[
a_i a_j = \sum_k c_{ij}^k a_k, \ \ \ b_rb_s = \sum_t d_{rs}^t b_t
\]
Writing $f_{ir} = a_i \otimes b^*_r$, the 
Write $\eta : A \to F(A,B)\otimes B$ in terms of these bases so that $\eta (a_i)$ is an $n\times n$ matrix 
\[
\eta (a_i) = \sum_r f_{ir}B_r
\]
with coefficients in $F(A,B)$. The identities required of $F(A,B)$ are given by comparing the matrix entries:
\begin{align}
( \sum_r f_{ir}B_r )( \sum_s f_{js}B_s)& = \eta(a_i )\eta(a_j ) \\
		&= \sum_k c_{ij}^k \eta(a_k) \\
		& = \sum_{k \ r}c_{ij}^k f_{kr}B_r
\end{align}
This can be computed by matrix multiplication, taking care not to assume commutativity of the $f_{ij}$. This method is particularly useful in calculations where $A = B=\mathbb{Q}[x]/p $. It will also be helpful in understanding and  computing with the module version of this construction.

\subsection{Calculating $F(\mathbb{Q}[x]/p,\mathbb{Q}[x]/p)$ over $\mathbb{Q}$}
Here $k = \mathbb{Q}$ and $p = p(x) =  \sum p_ix^i$ is a polynomial. The problem of computing $F(\mathbb{Q}[x]/p,\mathbb{Q}[x]/p)$ can be broken into three stages as follows.
\begin{proposition} \label{prop:Qcalc}
\begin{enumerate}
	\item $F(\mathbb{Q}[x],\mathbb{Q}[x]) = T(\mathbb{Q}[x]^\circ)$.
	\item $F(\mathbb{Q}[x],\mathbb{Q}[x]/p) = T((\mathbb{Q}[x]/p)^\circ)$.
	\item $F(\mathbb{Q}[x]/p,\mathbb{Q}[x]/p) = T((\mathbb{Q}[x]/p^\circ))/J$ where $J$ is the ideal generated by 	
\[p_0\epsilon(\alpha) + \sum p_i\triangle ^{i-1}\alpha
\]
for $\alpha \in (\mathbb{Q}[x]/p)^\circ$
\end{enumerate}
\end{proposition}
\begin{proof}
Yoneda's lemma provides a strategy for proving all three parts. For algebras $A, \ B, C$, from \ref{thm:actions} we have $Alg(F(A,B),C) \cong Alg(A, [B^\circ,C])$. Yoneda's lemma says that  $Alg(F(A,B),C) \cong Alg(A, [B^\circ,C]) \cong  Alg(F', C)$ (with the isomorphisms being natural) for all algebras $C$ if and only if $F' = F(A,B)$.  Therefore it is sufficient to show that
\begin{itemize}
	\item Step 1. $Alg(\mathbb{Q}[x],[\mathbb{Q}[x]^\circ, C]) \cong Alg(T(\mathbb{Q}[x]^\circ), C)$
	\item Step 2. $Alg(\mathbb{Q}[x],[(\mathbb{Q}[x]/p)^\circ, C]) \cong Alg(T((\mathbb{Q}[x]/p)^\circ), C)$
	\item Step 3. $Alg(\mathbb{Q}[x]/p,[(\mathbb{Q}[x]/p)^\circ, C]) \cong Alg((T(\mathbb{Q}[x]/p)^\circ)/J, C)$.
\end{itemize}
\emph{Step 1}. An algebra homomorphism $\theta:\mathbb{Q}[x] \to [\mathbb{Q}[x]^\circ, C]$ is determined by the image of $\theta(x):\mathbb{Q}[x]^\circ \to C$ with no restrictions. 

Similarly an algebra homomorphism $\tilde{\theta}: T(\mathbb{Q}[x]^\circ) \to C$ is determined by its restriction to $T^1(\mathbb{Q}[x]^\circ) = \mathbb{Q}[x]^\circ$.  There are no restrictions; any linear map $\tilde{\theta}$ gives rise to an algebra homomorphism $\tilde{\theta}: T(\mathbb{Q}[x]^\circ) \to C$.  

The correspondence $\theta \leftrightarrow \tilde{\theta}$ thus establishes a (natural) correspondence  $Alg(\mathbb{Q}[x],[\mathbb{Q}[x]^\circ, C]) \cong Alg(T(\mathbb{Q}[x]^\circ, C)$ and the first part of the proposition.

\emph{Step 2} The argument is identical, replacing $\mathbb{Q}[x]^\circ$ with  $\mathbb{(Q}[x]/p)^\circ$.

\emph{Step 3} Observe that $\theta:\mathbb{Q}[x]/p \to [(\mathbb{Q}[x]^/p)^\circ, C]$ is an algebra homomorphism if $\theta(f(x)p)=0$ for any product  $f(x)p$ in the ideal generated by $p$. That is, $\theta(f(x)p) = 0$ if and only if for any $\alpha$ in $(\mathbb{Q}[x]^/p)^\circ$
\[
<\alpha, f(x)p> = \sum_{(\alpha)}<\alpha_{(2)},f(x)><\alpha_{(1)},p> = 0.
\]
Therefore, $\theta(f(x)p)=0$ if and only if for any $\alpha$,
\begin{align}
<\alpha,\theta(p)> & =   \sum^{n-1}_{i=0} p_i<\alpha, \theta(x^i)> \\
		& =  p_0\epsilon(\alpha) +p_1<\alpha, \theta(x)> + \sum_{k=2}^{n-1}\sum_{(\alpha)} <\alpha_{(k)}, \theta(x)>...<\alpha_{(1)},\theta(x)>
\end{align}
On the other hand, $\tilde{\theta}:T(\mathbb{Q}[x]/p)^\circ) \to C$ vanishes on $J$ if and only if 
\begin{align}
0 & = \tilde{\theta}(p_0\epsilon(\alpha) + p_1\alpha + \sum_{i=2}^{n-1} \sum_{(\alpha)} \alpha_{(i)} \otimes ... \otimes \alpha_{(1)}) \\
 & = p_0\epsilon(\alpha)+p_1\tilde{\theta}(\alpha) +\sum_{i=2}^{n-1} p_i\sum_{(\alpha)}\tilde{\theta}(\alpha)_{(i)}... \tilde{\theta}(\alpha)_{(1)} \\
 & =  p_0\epsilon(\alpha) +p_1<\alpha, \theta(x)> + \sum_{k=2}^{n-1}\sum_{(\alpha)} <\alpha_{(k)}, \theta(x)>...<\alpha_{(1)},\theta(x)>
\end{align}
as required.
\end{proof}

\subsection{Extensions of $\mathbb{Q}[x]/p$}
Let $A = \mathbb{Q}[x]/p(x)$. While it is inconvenient to calculate $F(A,A)$ in its entirety, extensions $\theta:A \to k\otimes A$ for a field $k$ correspond to algebra homorphisms $F(\theta): F(A,A) \to k$.  These can be calculated as follows.

The algebra $\End(A)$ can be identified with $M_n(\mathbb{Q})$ via the basis of powers of $x$,  $\{ x^i\}$. Left multiplication defines a representation of $L:A \to M_n(\mathbb{Q})$ in which $x$ is represented by its companion matrix $C$.  Denote by $<C>$ the subalgebra of $M_n(\mathbb{Q})$ generated by $C$.
We can further identify $k\otimes A= k[x]$ with its image in $M_n(k)$ for $k$ a field containing $\mathbb{Q}$, so that if $Y$ is any matrix in $<C>$, the first column of $Y$ gives the coefficients $Y = \sum y_i X^i$.

Now assume the roots of $p$ are all distinct and are contained in the field $k$, and choose an ordering $\{\mu_1, ... , \mu_n\}$. The Vandermonde matrix $V$ expresses roots in terms of the basis $\{x^i\}$, so that if $V = (v_{ij})$, ${\mu_j = \sum v_{ij} x^i}$. Conjugation by $V$ diagonalises $k\otimes A$ considered as a subalgebra of $M_n(k)$, identifying $k\otimes A$ with the subalgebra $D_n(k)$ of diagonal matrices. 

Now let $[n]$ denote the set of integers from 1 to {n}, and let $\sigma: [n] \to [n]$ be any function.  Then $\sigma$ induces a homomorphism $\tilde{\sigma} :D_n(k) \to D_n(k)$ simply by repositioning the entries. These homomorphisms form a monoid which includes all of the symmetric group $S_n$ as its invertible elements.

Conjugating $\tilde{\sigma}$ by Vandermonde matrices provides a homomorphism
\[
W_\sigma  = V^{-1} \tilde{\sigma} V: k \otimes A \to k \otimes A
\]
or, by restriction, an extension $W_\sigma : A \to k\otimes A.$ 

Consider the matrix representing $W_\sigma$ in terms of the basis $\{1, C, C^2,...,C^{n-1}\}$. With respect to this basis writing $W_\sigma = \sum w_iC^i$, the coefficients $w_i$ are exactly the first column of $W_\sigma$. 

Write $\eta = \eta(A,A):A \to F(A,A)$ as $\eta (x) = f_0 + f_1x + f_2x^2 + ... + f_{n-1}s^{n-1}$. The homomorphism $F(W_\sigma): F(A,A) \to k$ is given by $F(f_i) = w_i$

One can ask if the homomorphisms $W_\sigma$ correspond to Galois transformations of $k$.  For each root $\mu$ there is a homomorphism $\pi_\mu: \mathbb{Q}[x]/p(x) \to k$ sending $x \mapsto \mu$. The requirement is that $\pi_{\bar{\sigma}(\mu)}(x) = \pi_{\mu}W_\sigma (x)$ for all roots $\mu$.

\subsection{Representations of $F(\mathbb{Q}[x]/p,\mathbb{Q}[x]/p)$ in $M_n(k[\lambda])$}. 
The loop algebra representations of $F(\mathbb{C},\mathbb{C})$ have analogues for $\mathbb{Q}[x]/p$. With notation as in the previous subsection, let $Z$ be any matrix in $M_n$, and let $g(u) = exp(uZ)$. Then $g(u)Cg(-u)$ also satisfies the polynomial $p$. 

Expand $g(u)Cg(-u)$ as the exponential of $ad(Z)$:
\begin{align}
g(u)Cg(-u) & =  exp(ad(uZ))C\\
		& = C + u[Z,C] + \frac{1}{2} u^2 [Z,[Z,C]] + ...
\end{align}
This expression satisfies the polynomial $p$ \emph{for any indeterminate t}.  Now let $u = \lambda x$. Then
\begin{align}
g(u)Cg(-u) & =  exp(x.ad(\lambda Z))C\\
		& = C + [\lambda Z,C]x + \frac{1}{2} [\lambda Z,[\lambda Z,C]]x^2  + ...
\end{align}
which can be considered as an element in $M_n(\mathbb{Q}[\lambda]) \otimes \mathbb{Q}[x]/p$.
 
 The assignment $x \mapsto g(u)C(g-u)$ written out this way then provides an extension homomorphism
  $\sigma_Z:\mathbb{Q}[x]/p \to M_n(\mathbb{Q}[\lambda]) \otimes \mathbb{Q}[x]/p$ as promised. The corresponding map $F(\sigma)$ represents $F(\mathbb{Q}[x]/p,\mathbb{Q}[x]/p)$ in $M_n(k[\lambda])$ in $M_n(\mathbb{Q}[\lambda]) $.

\subsection{$F(B, B)$ for $B = k[d]/d^2$, and relation to the Pareigis Hopf algebra}

By analogy with $F(\mathbb{C},\mathbb{C})$, the algebra $F(B,B)$ for $B = k[d]/d^2$, The algebra $F(B,B)$ is generated by two elements, $g_0, \ g_1$, such that 
\[
\eta: B \to F(B,B) \otimes B,  \ \ \ \eta(d) = g_0\otimes 1 + g_1 \otimes d.
\] 
The same reasoning gives explicit defining identities,
\[
g_0g_1 = g_1g_0 = 0 = g_0^2
\]
and an explicit basis
\[
F(B,B) = <g_1^i, g_0g_1^i: i \geq 0>.
\]
Comultiplication is determined: since $1\otimes \eta \circ \eta = \triangle \otimes 1 \circ \eta$,
\[
\triangle g_1 = g_1\otimes g_1, \ \ \ \triangle g_0 = g_1 \otimes g_0+ g_0 \otimes 1.
\]
The counit $\epsilon: F(B,B) \to k$ is determined by the evaluation map $B\otimes B^* \to k$, so that
\[
\epsilon(g_0) = 0, \epsilon (g_1) = 1.
\]

This bialgebra is closely related to a well understood Hopf algebra, first presented by Pareigis \cite{pareigis1980hopfalg}. 

\begin{definition} \label{pareigis}
There is a non-commutative, non-cocommutative Hopf algebra $H$, the Pareigis Hopf Algebra, defined over any field $k$, which has presentation given by generators and relations:
\begin{enumerate}
    \item $H = k[x, y, 1/y]/(xy + yx, x^2)$ as an algebra (1/y is a left and right multiplicative inverse for y).
    \item The antipode $s: H \to H$ is defined by $y \mapsto 1/y, x \mapsto xy$
    \item $\triangle : H \to H \otimes H, y \mapsto y \otimes y, x \mapsto x \otimes 1 + (1/y) \otimes x$
    \item $\epsilon: H \to k, x \mapsto 0, y \mapsto 1$
\end{enumerate}
\end{definition}

The relationship of $H$ with $F(B,B)$ is as follows. Let $H^-$ be the subalgebra of $H$ generated by $x$ and $1/y$. This is a bialgebra, but not a Hopf algebra, as it is not closed under the antipode.

\begin{proposition} There is a homomorphism 
\[
    p:F(B,B) \to H
\]
establishing an isomorphism of $F(B,B)$ with $H^-$ as bialgebras.
\end{proposition}

\begin{proof}
Set $p(g_0) = x, \ p(g_1) = 1/y.$ and observe that the generators and relations coincide.
\end{proof}

\section{Comodules}

A (left) \emph{comodule} $X$ over a coalgebra $H$ is a vector space with a comultiplication
\[
\triangle_X : X \to H\otimes X,  \ \ \  \triangle_X (d) = \sum_{(d)} d_{(1)}\otimes d_{(0)}
\]
satisfying coassociativity:
\[
(1_C \otimes \triangle_X) \circ \triangle_X = (\triangle_C \otimes 1_X) \circ \triangle_X : X \to C\otimes C \otimes X
\]
and counit
\[
(\epsilon \otimes  1_X) \circ \triangle_X = 1_X : X \to k\otimes X = X.
\]

When it is helpful, a left comodule $X$ may be denoted $_HX $ to emphasise that $X$ is an $H$ comodule. Similarly, left modules $M$ over $A$ may be denoted $_AM$.   

If $_HD_1, \ _HD_2$ are two comodules, a linear map $\zeta : D_1 \to D_2$ is a map of $H$-comodules if
\[
\triangle_2 \circ \zeta = 1_H \otimes \zeta \circ \triangle_1: D_1 \to H \otimes D_2.
\]

Denote by $_HComod$ the category of (left) $H$-comodules.

\subsubsection*{Example: Duals of modules} The dual coalgebra model provides examples of comodules. Let $V$ be a (finite dimensional) vector space on which a finite dimensional algebra $S$ acts on the right. Then dualising the action gives rise to a comultiplication
\[
\triangle :V^* \to (V\otimes S)^* = S^* \otimes V^*
\]

\subsubsection*{Example: Comodules of $F(B,B)$ for $B= k[d]/d^2$} 
Let $M = \oplus M_i$ be a chain complex with $M_i = 0$ for $i \geq 1$ with boundary map $d:M_i \to M_{i+1}$.
 Define a comodule structure on $M$ by setting
\[
\rho: M \to F(B,B)\otimes M, \ \ \ m_i \mapsto g_1^i \otimes m_i + g_0g_1^{i+1}\otimes dm_i.
\]
That $M_i$ is a comodule can be verified directly. 

Conversely, given an $F(B,B)$ comodule $Z$, set $\widetilde{Z}_k$ to be the subspace of $Z$ spanned by elements $z_{(1,i)}, \ z_{(0,i)}$ in the expression 
\[
\triangle z = \sum_i g_{i}^{-i} \otimes z_{(1,i)} + g_0g_1^{-i+1}\otimes z_{(0,i+1)}
\]

This connection provides homological links with extensions more generally. Let $K$ be a coalgebra and $Z$ a $K$ comodule.  A coalgebra homomorphism $\kappa: F(B,B) \to K$ gives $Z$ the structure of an $F(B,B)$ comodule, and hence defines a chain complex $\widetilde{Z}$. Suppose $A$ is a finite dimensional algebra and let $V$ be a finite dimensional $A$ module. The theme of the present section is the construction of a universal $F(A,A)$ comodule $D(V,V)$.  In this way, coalgebra homomorphisms $\rho: F(B,B) \to F(A,A)$ provide a source of chain complexes. Thus the coalgebra structure of $F(A,A)$ is closely linked to the homological information pertaining to $A$.

This relationship between the Pareigis Hopf algebra and homologies are developed extensively (and correctly in the differential graded setting) in \cite{anel2013sweedler}.

\subsection{Measuring comodules}
\begin{definition} Let $A$, $B$ be algebras, let $_AM, \ _BN$ be $A$ and $B$ modules respectively.  Let $\rho:H \to \Hom_k(A,B)$ be a measuring coalgebra, and let $_HX$ be an $H$ comodule.  Say that $\gamma: X \to \Hom_k(M,N)$ \emph{measures} if 
\[
\gamma(x)(am) = \sum_{(x)} <\rho(x_{(1)}),a> <\gamma (x_{(0)}), m>
\]
for all $x$ in $X$, $a$ in $A$, and $m$ in $M$.
\end{definition}

As with measuring coalgebras (proposition \ref{prop:meas}), there is an equivalent condition.

\begin{proposition} \label{prop:meascomod}
$\gamma: _HX \to \Hom_k(_AM, _BN)$ measures if and only if the corresponding map
\[
\gamma:_AM \to \Hom_k(_HX, _BN) 
\]
is a map of $A$-modules.
\end{proposition}

As for algebras and coalgebras, use the notation
\[
\Hom_k(_HX,_BN) := [X,N] = _{[H,B]}[X,N]
\]
to emphasise the module structure.

\begin{proof}
$\Hom_k(_HX, _BN)$ is an $[H,B]$ module with the action given by 
\[
<\alpha * \zeta ,x> = \sum_{(x)} <\alpha,x_{(1)}><\zeta,x_{(0)}>
\]
for $\alpha$ in $[H,B]$, $\zeta$ in $[X, N]$ and $x$ in $X$. The measuring map $H \to \Hom_k(A,B)$ corresponds to an algebra homomorphism $A \to [H,B]$. With respect to this map, $[X,N]$ becomes an $A$-module.  
\end{proof}.

\subsection{Examples of measuring comodules} 
\subsubsection*{Example: Module maps}
Let $kf$ be the one dimensional comodule over the one dimensional coalgebra $ke$ and let $\triangle f = e \otimes f$. Suppose $\rho : ke \to \Hom(A,B)$ measures (so that $\rho(e)$ is a homomorphism).  A linear map $\theta : kf \to \Hom (_AM,_BN)$ is a measuring map if and only if $\theta(f)$ is a module map with respect to the homomorphism $\rho(e)$.

\subsubsection*{Example: Modules of extensions}\label{example:extmod}
Let $\psi: A \to S\otimes B$ is an extension, Let $M, \ N, \ V$ be modules over $A, \ B, S$ respectiviely and suppose that $V$ is finite dimensional. Suppose also that there is a map $\theta: M \to V\otimes N$ which is an $A$ module map with the $A$ module structure on $V \otimes N$ that conferred by $\psi$. That $V^*$ is an $S^\circ$ module can be seen by dualising the maps defining the action of $S$ on $V$ so that for $v^*$ in $V^*$,  $s$ in S and $w$ in $V$,
\[
<v^*, sw> = \sum_{(v^*)} <v^*_{(1)}, s><v^*_{(0)}, w>.
\] 
It is helpful to extend the application of Sweedler notation to express elements in tensor products. Set
\begin{align}
\psi(a) = & \sum_{(\psi(a))} \psi (a)_{(1)}\otimes \psi(a)_{(0)} \in S \otimes B\\
\theta(m) = & \sum_{\theta(m)} \theta(m)_{(1)} \otimes \theta(m)_{(0)} \in V \otimes N.
\end{align}
The claim is that map $\theta$ corresponds to a map
\[
\theta^*:V^* \to \Hom(M, N), \ \ <\theta^* v^*, m> := <v^*,\theta(m)_{(1)}>\theta(m)_{(0)}.
\]
is a measuring map, just as the map $\psi$ corresponds to a measuring map (of algebras)
\[
\psi^* : S^\circ \to \Hom(A,B), \ \ <\psi^*s^*,a> = \sum_{(\psi(a))} <s^*,\psi(a)_{(1)}>\psi(a)_{(0)}
\]
The statement that $\theta$ is a module map written out in Sweedler notation is
\begin{align}
\theta(am) = & \psi(a)\theta(m) \\
	= & \sum_{(\psi(a))(\theta(m))} \psi(a)_{(1)}\theta(m)_{(1)} \otimes \psi(a)_{(0)}\theta(m)_{(0)}.
\end{align}
Then $\theta^*:V^* \to \Hom(M, N)$ measures if
\[
<\theta^*v^*,am> = \sum_{(v^*)} <\psi^*(v^*_{(1)}),a><\theta^*(v^*_{(0)}),m>
\]
With the notation in place, this is a matter of direct calculation:
\begin{align}
<\theta^*v^*,am>  = &  \sum_{(\theta(am))} <v^*,\theta(am)_{(1)}>\theta(am)_{(0)} \\
	= & \sum _{(\psi(a))(\theta(m))} <v^* ,\psi(a)_{(1)}\theta (m)_{(1)} >\psi(a)_{(0)}\theta(m)_{(0)} \\
	= & \sum_{(v^*)(\psi(a))(\theta(m))} <v^*_{(1)},\psi(a)_{(1)}><v^*_{(0)},\theta(m)_{(1)}>\psi(a)_{(0)}\theta(m)_{(0)} \\
	= & \sum_{(v^*)}<\psi^*v^*_{(1)},a><\theta^*v^*_{(0)},m> 
\end{align}
as required.
\subsubsection*{Example: Connections}
A familiar example of measuring comodules arises in differential geometry as connections. Let $A$ be an algebra of (differentiable) functions on a manifold $Y$, and let $Der$ be the Lie algebra of derivations of $A$, that is, the vector fields on $Y$. Then define $H$ to be $Der \oplus \mathbb{R}e$, where $\mathbb{R}e$ is a vector space, and endow $H$ with a coalgebra structure by setting
\[
\triangle e= e\otimes e, \ \  \triangle \eta = \eta \otimes e+ e \otimes \eta, \ \ \epsilon (e) = 1, \ \ \epsilon (\eta) = 0.
\]
for $X$ in $Der$. The linear map $\rho: H \to \Hom_k(A,A)$ defined by
\[
\rho (e) = 1_A, \ \  \rho(X) = \eta \in Der \subset \Hom_k(A,A)
\]
for $X$ in $Der$ measures.

Sections $S$ of a bundle $E$ over $Y$ is just an $A$ module. $H$ can be considered a comodule over itself. 
If $\gamma$ is a connection, $\gamma$ is a linear map 
\[
\gamma: Der  \to \Hom(_AS, _AS).
\]
satisfying
\[
\gamma (X) (as) = (Xa)s + a\gamma(X)s
\]
for $a \in A$, $s \in S$, 
Extend $\gamma$ to all of $H$ by setting $\gamma (e) = 1_S$. Thus
\[
\gamma : H \to \Hom(S,S)
\]
The condition for $\gamma$ to be a connection can be rewritten
\[
<\gamma(X), as> = <\rho(X),a><\gamma(e),s> + <\rho(e),a><\gamma(X),s>
\]
which is just the measuring condition.

\subsection{$\widecheck{Mod}$ and $\widehat{Mod}$}

Just as measuring coalgebras for a fixed pair of algebras $A, \ B$ can be compared by maps of measuring coalgebras, measuring comodules for a pair $_A M, \ _B N$ can be compared.

\begin{definition} \label{def:mapmeascomod}
A map of comodules $\zeta :X \to X'$ is a \emph{map of measuring comodules} if $\zeta$ respects the measuring maps:
\[
\xymatrixcolsep{5pc}
\xymatrix{
X \ar[r] \ar[d]_\zeta & \Hom_k(_AM,_BN) \\
X' \ar[ur] &
}
\]
\end{definition}

The concept of extensions of algebras has a corresponding concept for modules over algebras.  
\begin{definition} \label{def:modext}
If $_AM, \ _BN, \ _ST $ are modules over $A, \ B, S$ respectively, an \emph{extension of modules} is a map 
\[
\zeta: M \to T \otimes N
\]
which is a module map with respect to an extension $\psi: A \to S \otimes B$.  If $_{S'}T'$ is an $S'$ module, and $\zeta: M \to T' \otimes N$ another extension with respect to an algebra extension $\psi':A \to S'\otimes B$, a map of extensions $\zeta \to \zeta'$ consists of a map of (algebra) extensions $\theta: S\to S'$ and a map of $S$ modules $\upsilon:T \to T'$ so that the following diagram commutes.
\[
\xymatrixcolsep{5pc}
\xymatrix{
M \ar[r]^\zeta \ar[dr]_{\zeta'} & T \otimes N\ar[d]^ {\upsilon \otimes 1_N}\\
	& T' \otimes N
}
\]
\end{definition} 

The categories $\widecheck{Alg}$ and $\widehat{Alg}$ have analogues in modules. 
\begin{proposition}
Let $M, \ N$ be modules over algebras $ A, \ B $ respectively.
\begin{enumerate}
\item There is a category $\widecheck{Mod}(M,N)$ whose objects are measuring comodules $X \to \Hom_k(M,N)$ and whose morphisms are maps $\kappa: X \to X'$ of measuring comodules.
\item There is a category $\widehat{Mod}(M,N)$ whose objects are module maps $M \to _ST\otimes N$ for some $\sigma: A \to S\otimes B$ in $\widehat{Alg}(A,B)$ and whose morphisms are maps $\nu:T \to T'$ of extensions.
 \end{enumerate}
\end{proposition}
\begin{proof}
The procedure to prove this is identical to that of the corresponding proof for algebras. Ultimately, as before, the proof rests on the fact that for both $\widecheck{Mod}$ and $\widehat{Mod}$, composition of horizontal maps in both categories is simply the tensor product of comodules and modules respectively.
\end{proof}
\begin{theorem}
$\widecheck{Mod}$ and $\widehat{Mod}$ are bicategories 
\end{theorem}

Again, the proof is parallel to the proof for the corresponding statement for algebras.

The example \ref{example:extmod} enables enables a comparison of the subcategories $\widecheck{Mod}_f$ and $\widehat{Mod}_f$ where the morphisms are restricted to finite dimensional measuring comodules and finite dimensional modules $_SU$ respectively. This comparison is made formal in the following equivalence.

\begin{proposition}
There is a natural equivalence of categories
\[
*: \widecheck{Mod}_f \cong \widehat{Mod}_f^{co}:*
\]
\end{proposition}
The proof follows the same steps as \ref{thm:dual}.

\subsection{Universal measuring comodules}
Comodules share the same profound finiteness property that coalgebras have: \emph{every element of a comodule is an element of a finite dimensional subcomodule} (Sweedler, Chapter II \cite{Swee1969}).  As with coalgebras, a consequence is that the category of comodules over a given coalgebra has colimits and coequalisers, ensuring the existence of universal measuring comodules.

\begin{definition}
Let $_AM, \ _BN$ be algebras. The \emph{universal measuring comodule} $Q(M,N)$ is the $P(A,B)$ comodule with a measuring map $\psi: Q(M,N) \to \Hom_k(M,N)$ with the property that if $\gamma:X\to \Hom_k(M,N)$ is any measuring map over any measuring coalgebra $\rho: H \to \Hom_k(A,B)$, then there is a unique map of measuring comodules  $Q(\gamma): X \to Q(M,N)$ such that the following diagram commutes.
\[
\xymatrixcolsep{5pc}
\xymatrix{
Q(M,N) \ar[r]^\psi & \Hom_k(M,N) \\
X\ar[ur]_\gamma \ar[u]^{Q(\gamma)}&
}
\]

\end{definition}

Notice that by the universal property for the universal measuring coalgebra, any measuring comodule over a measuring coalgebra $H$ can be considered as a measuring comodule over $P(A,B)$.  

$Q(M,N)$ exists. The finiteness properties of comodules enable $Q(M,N)$ to be constructed as a coequaliser in the same manner as the proof outlined in \ref{prop:P(A,B)exists}. A more detailed study of these structures can be found in \cite{hyland2017hopf}.

The fundamental property of the universal measuring comodule is the analogue of proposition \ref{prop:natequiv}.

\begin{proposition}\label{prop:natequivmod}
Let $M, \ N$ be modules over algebras $A, \  B$ respectively.
\[
_{P(A,B)}Comod(-, Q(M,N)) \cong _AMod(M, [-,N]).
\]

\end{proposition}

\begin{proof}
Let $X$ be a $P(A,B)$ comodule. First observe that $[X,N]$ is a $[P(A,B),B]$ module. For $\alpha \in [P(A,B),B]$ and $\nu \in [X,N]$, define the action by setting 
\[
<\alpha * \nu, x> = \sum_{x} <\alpha, x_{(1)}><\nu,x_{(0)}>.
\]
Then $[X,N]$ becomes an $A$ module in the usual way via the homomorphism $A \to [P(A,B),B]$ of theorem \ref{thm:F(A,B)}. As with the corresponding statement for algebras (proposition \ref{prop:natequiv}), naturality and equivalence must be checked, but present no exceptional problems.

\end{proof}

The universal measuring comodule enables the definition of a global enriched category of modules.
\begin{proposition}
There is an enriched category $\underline{Mod}$ whose objects are modules and whose morphisms are the universal enveloping comodules $Q(M,N)$.
\end{proposition}
\begin{proof}
The slightly uncomfortable fact that requires verification is that an element $\alpha$ in $Q(_AL,\ _BM)$ can compose with an element $\beta$ in $Q(_BM,\ _CN)$ to give an element in $Q(_AL,\ _CN)$, but that is a consequence of the universal property, as follows.  Tensoring followed by composition provides a map
\[
Q(M,N) \otimes  Q(L, M) \to \Hom(M,N) \otimes \Hom(L,M) \to \Hom (L,N)
\]
By the universal map $P(B,C) \otimes P(A,B) \to P(A,C)$, $Q(M,N) \otimes  Q(L, M)$ can be considered a $P(A, C)$ comodule. That the composition of maps above is a measuring map can be verified directly.
\end{proof}

Thus $\underline{Mod}$ is a global category of modules: a map between an $A$ module $M$ and a $B$ module $N$ is simply an element of the universal measuring comodule $Q(M,N)$; there is no need to specify an algebra homomorphism relating $A$ and $B$.

This global category is fibred over a 2-category $\widecheck{Alg1}$ of algebras adapted from the bicategory $\widecheck{Alg}$ as follows. 

First, for algebras $A, \ B$ define a category $\mathcal{P}(A,B)$ whose objects are are subcoalgebras of the measuring coalgebra $P(A,B)$ and whose morphisms are inclusions, that is, if $H_1, \ H_2$ are subcoalgebras of $P(A,B)$ the morphism set $\mathcal{P}(A,B)(H_1, H_2)$ contains a single element, $i_{12}$ if $H_1 \subset H_2$ or is the empty set. Composition of morphisms, the associativity and unit properties are immediate. 

For the bicategory $\widecheck{Alg1}$, objects are algebras and the set of morphisms $\widecheck{Alg1}(A,B)$ is $\mathcal{P}(A,B)^{op}$, 

Composition $H_2\circ H_1$ in $\widecheck{Alg1}$ is defined by setting 
\[
H_2 \circ H_1 = P(i_2 \otimes i_1)(H_2 \otimes H_1)
\]
the image of $H_2 \otimes H_1$ in $P(A_1, A_3)$.

\begin{proposition}
$\widehat{Alg1}$ defined above is a 2 category.
\end{proposition}
\begin{proof}
Associativity can be checked by inspecting the following diagram.
Let $H_i:A_i \to A_{I+1}, i = 1,\ 2, \ 3$ be morphisms (subcoalgebras of $P(A_i, A_{i+1})$). 
Consider the diagram
\[
\xymatrixcolsep{5pc}
\xymatrix{
H_3\otimes H_2 \otimes H_1 \ar[r] \ar[d] & H_3 \otimes H_2\circ H_1 \ar[d]\\
H_3 \circ H_2 \otimes H_1 \ar[r] & P(A_1,A_3)
}
\]
The critical observation is that any element in $H_3\circ H_2 \subset P(A_2, A_4)$ is in the image of an element in $H_3 \otimes H_2$. Therefore for $k \in H_3\circ H_2 \otimes H_1$, there is an element $h \in H_3\otimes H_2 \otimes H_1$ with $q(h) = k$. By the uniqueness of the universal map, it must be that $s(k) = r(p(h))$. Thus $s(H_3\circ H_2 \times H_1) \subset r(H_3 \otimes H_2 \circ H_1)$.  Similarly $r(H_3 \otimes H_2\circ H_1) \subset s(H_3\circ H_2 \times H_1) $.

The interchange law requires more unravelling than proof. It is satisfied since if $H_i, \subset  J_i  \subset K_i$ are subcoalgebras of $P(A_i, A_{i+1})$ the composition of inclusions $H_2\circ H_1 \subset J_2 \circ J_1 \subset K_2 \circ K_1$ is just the composition $H_2\circ H_1 \subset K_2 \circ K_1 = (H_2 \subset J_2 \subset K_2) \circ (H_1 \subset J_1 \subset k_1)$.
\end{proof}
The fibred structure of $\underline{Mod}$ can now be given.

\begin{proposition}
There is a functor $\Pi: \underline{Mod} \to \widecheck{Alg1} $ giving $\underline{Mod}$ the structure of a fibred category.
\end{proposition}
\begin{proof}
The steps in the proof are as follows.
\begin{enumerate}
	\item Construct the projection functor $\Pi: \underline{Mod} \to \widecheck{Alg1}$, and verify that it is a functor.
	\item For an element $d\in Q(M,N)$ construct the "pullback" $d^*(N)$ of $d$.
	\item Construct an $A$ module map $j: M \to d^*(N)$
	\item Construct an element $e: d^*(N) \to N$ with $\Pi (e) = \Pi(d)$ and $e \circ j = d$.
\end{enumerate}
1. \emph{The functor $\Pi$.} For an element $d$ in $Q(M,N)$ let $D$ be the minimal measuring subcomodule containing $d$, (which will be finite dimensional) and let $\Pi(d)$ be the minimal subcoalgebra of $P(A,B)$ over which $D$ is a subcomodule.  $\Pi(d)$ is also finite dimensional. 

$\Pi$ is functorial in the sense that that there is a natural transformation, 
\[
\kappa: \Pi (-) \circ \Pi(-) \to \Pi(-\circ -).
\]
(It is at this point that we need the morphism to be $\mathcal{P}(A,B)^{op}$ rather than $\mathcal{P}(A,B)$.) This is a consequence of uniqueness in the universal property of $Q(-,-)$ as follows.  By their respective universal properties there are unique maps
\[
Q(M_2, M_3) \otimes Q(M_1, M_2) \to Q(M_1, M_3), \ \ P(A_2, A_3) \otimes P(A_1, A_2) \to P(A_1, A_3)
\]
The image of $\Pi(d_2) \otimes \Pi (d_1)$ is a measuring subcoalgebra of $P(A_1, A_3)$ and 
\[
\triangle (d_2 \circ d_1) = \triangle (d_2) \otimes \triangle (d_1).
\]
Thus $\Pi(d_2) \otimes \Pi (d_1) \to \Hom(A_1, A_3)$ measures and its imagine in $P(A_1,A_3)$ is $\Pi(d_2) \circ \Pi (d_1)$.  But computing $\triangle (d_2 \circ d_1)$ as above, it must be that $\Pi(d_2 \circ d_1 ) \subset P(\Pi(d_2) \otimes \Pi(d_1)$ since $\Pi(d_2 \circ d_1)$ is minimal, or otherwise put, there is a natural transformation
\[
\Pi(d_2)\circ \Pi(d_1) \to \Pi(d_2 \circ d_1)
\]
as required.

2. \emph{The pullback of $d$}. Define $d^*(N) = : \Pi(d)\otimes N$.  The $A$ action on $\Pi(d)\otimes N$ is given by
\[
\xymatrixcolsep{3pc}
\xymatrix{
A\otimes \Pi(d)\otimes N \ar[r]^{1\otimes \triangle \otimes 1} & A\otimes \Pi(d) \otimes \Pi(d) \otimes N \ar[r]^z &\Pi(d)\otimes B \otimes N \ar[r] &\Pi(d)\otimes N
}
\]
where $z$ is evaluation on the second factor:
\[
z(a\otimes h_2 \otimes h_1 \otimes n) = h_2 \otimes <a, h_1> \otimes N.
\]

3. \emph{The module map $j: M \to d^*(N)$}. Define $j: M \to d^*(N) $ be the composition
\[
\xymatrixcolsep{3pc}
\xymatrix{
M \ar[r] & d \otimes M \ar[r] &\Pi(d) \otimes Q(M,N)\ar[r]& \Pi(d) \otimes N
}
\]
That this is a map of $A$ modules can be verified directly.

4. For any finite dimensional coalgebra $K$, the dual algebra $K^*$ acts on $K$:
\[
\xymatrixcolsep{3pc}
\xymatrix{
K^* \otimes K\ar[r]^{1\otimes \triangle} & K^* \otimes K\otimes K \ar[r]^{\tau \otimes 1} & K\otimes K^*\otimes K \ar[r]^{\ \ \ \ 1\otimes <-,->} & K
}
\]
where second map twists the first two factors, and the third map evaluates the second and third factor.

Reversing these arrows gives $K^*$ the structure of a coalgebra:
\[
\xymatrixcolsep{3pc}
\xymatrix{
K^* \ar[r] & K^*\otimes K \otimes K^* \ar[r] & K \otimes K^* \otimes K^* \ar[r] & K \otimes K^*
}
\]
or with respect to a basis $\{k_i\}$, its dual basis $\{k_i^*\}$ and structure constants $\triangle k_i = \sum c_i^{pq} k_p \otimes k_q$, so that $k^*_mk^*_n = \sum_t c_t ^{mn} k*_t$
\[
\xymatrixcolsep{3pc}
\xymatrix{
k_i^* \ar[r] & \sum_m  k_i^* \otimes k_m\otimes k^m \ar[r] & \sum_m k_m \otimes k^*_i \otimes k^*_m \ar[r] & \sum_m c_t^{im}k_m \otimes k^*_t
}
\]
Now apply this to $K = \Pi(d)$, so that $\Pi(d)^*$ is a comodule over $\Pi(d)$. Then
\[
\Pi(d)^* \to \Hom(d^*(N), N), \ \ \beta \mapsto <\beta, ->\otimes 1_N:\Pi(d)*(N) = \Pi(d)\otimes N \to N
\]
is a measuring map. Thus $\epsilon\otimes 1_N \in \Pi(d)^*$ represents an element in $Q(\Pi^*(d), N)$ and $\Pi(\epsilon\otimes 1_N) = \Pi(d)$.

The observation that the composition $e \circ j$ is exactly $d$ completes the proof.
\end{proof}
\subsection{The Sweedler product for modules and comodules}
As $A\triangleleft H$ fills the role of tensor enabling a tensor - hom adjunction for the enriched category of algebras, so there is a similar construction $M\triangleleft X$ enabling a corresponding tensor - hom adjunction for the enriched category of modules. More precisely, the aim of this section is to define an $A\triangleleft H$ module  $M\triangleleft _HX$ such that there is a natural isomorphism
\[
Q(M\triangleleft X, N) \cong Q(M, [X, N])
\]

At first glance, it is not even immediately apparent that  the two sides are comodules with respect to the same coalgebra; the left hand side being a comodule over $P(A\triangleleft H, B)$ and the right hand side being a comodule over $P(A, [H,B])$. However, the natural isomorphism between these two is the first part of theorem \ref{thm:natequiven}, and the arguments to establish this result will follow along the same lines as for the corresponding result for algebras.
\subsubsection{Definition}
Let $H \to \Hom_k(A,B)$ measure. For an $H$-comodule $X$ and a $A$-module $M$ define $M\triangleleft X$ to be the coequaliser of the following diagram:
\[
\xymatrix{
A\otimes M \otimes X \ar[r] \ar[d] & A\otimes M \otimes H \otimes X \ar[r]  &A\otimes H \otimes M \otimes X \ar[d] \\
M\otimes X \ar[rr]  & &A\triangleleft H\otimes M \otimes X \ar[r]  & M\triangleleft X.
}
\]
 That is, set 
\[
M\triangleleft X = (A\triangleleft H \otimes M\otimes X)/R
\]
where R is the minimal $A\triangleleft H$ submodule such that the diagram commutes.

This construction has the same properties for modules as \ref{thm:actions} for algebras.

\begin{theorem}\label{thm:squaremodeq}
\[
\triangleleft: _AMod\times _{P(A,B)}Comod \to  _BMod
\]
\[
[-,-]: \  _{P(A,B)}Comod^{op} \times _BMod \to _AMod
\]
providing a tensor - hom adjunction
\[
Mod(M, [-,N]) \cong Mod(M\triangleleft -,N)
\]
where $M, \ N$ are modules over $A, \ B$ respectively.
\end{theorem} 
\begin{proof}
In essence, as with theorem \ref{thm:actions}, this is true because the statements are true for the underlying vector spaces. The steps required for verification are the same, the work required again reduces to brute computation.

\end{proof}

\subsection{Measuring comodules for comodules}

The concept of measuring comodules extends to measuring comodules for comodules, in parallel with measuring coalgebras for coalgebras. Again, the reason for introducing this concept is to complete a chain of three functorial equivalences parallel to theorem \ref{thm:natequiven}. Let $X, \ Y$ be comodules over $H, \ K$ respectively, and let $Z$ be a  $P(H,K)$ comodule. A map $\theta: Z \to \Hom_k(X,Y)$ measures if 
\[
\triangle <\theta (z), x> = \sum <\theta(z_{(1)}), x_{(1)}>\otimes <\theta(z_{(0)}),x_{(0)}>.
\]

Observe that the evaluation map
\[
P(H,K) \otimes H \to K
\]
is a coalgebra map.  As with the modules, this coalgebra homomorphism provides comodules over $P(H,K) \otimes H$ with a $K$ comodule structure. This in turn provides an alternative description of measuring comodules.

\begin{proposition}
With $X, \ Y, \ Z $ comodules over $H, \ K, \ P(H,K)$ respectively, a map  $\theta: Z \to \Hom_k(X,Y)$ is a measuring map if and only if the corresponding map
\[
\theta: Z \otimes H \to K
\]
is a map of comodules.
\end{proposition}.

As for measuring coalgebras, the strong finiteness property holds equally for comodules and ensures that colimits and coequalisers exist in the category of comodules (over a given coalgebra).  As for coalgebras, there is a terminal element for measuring comodules for comodules. The procedure for proving 

\begin{proposition}\label{prop;meascomod} 
For comodules $X, \ Y$ over a coalgebras $H, \ K$, there is an $P(H,K)$ comodule $Q(X,Y)$ and a measuring map $\ Q(X,Y) \to \Hom_k(X,Y)$ such that if $Z$ is another $P(H,K)$ comodule and $\zeta:Z \to \Hom_k(X,Y)$ measures, there is a unique map $Q(\zeta):Z \to Q(X,Y)$ such that the following diagram commutes.
\[
\xymatrixcolsep{5pc} 
\xymatrix{
Q(X,Y)  \ar[r] & \Hom_k(X,Y) \\
Z \ar[u]^{Q(\zeta)}\ar[ur]_{\zeta} &
}
\]
This establishes a natural equivalence
\[
Comod(-, Q(X,Y)) \cong Comod( -\otimes X, Y)
\]
\end{proposition}
\begin{proof}The proof of the existence of $Q(M,N)$ follows the same arguments as the corresponding proposition \ref{prop:P(A,B)exists} for algebras.  The natural equivalence can be shown applying the Yoneda and using proposition \ref{prop;meascomod}.
\end{proof}

\subsection{Enriched categories of modules and comodules}

As for the case of algebras, the existence of the universal measuring comodules enables a categories of modules enriched over comodules and of comodules enriched over comodules. The principle results about the universal measuring comodule are collected in the following theorem. 

\begin{theorem}
\begin{enumerate}
	\item There is a category $\underline{Mod}$ enriched over comodules whose objects are modules and whose morphisms $\underline{Mod}(_AM,_BN)$ is the measuring comodule $Q(M,N)$
	\item There is a category $\underline{Comod}$ enriched over comodules whose objects are comodules and whose morphisms $\underline{Comod}(_HX,_KY)$ is the universal measuring comodule $Q(X,Y)$
	\item Let $M,\ N$ be modules over $A, \ B$ respectively. There are natural equivalences
	\[
\underline{Comod}( -,Q(M,N)) \cong \underline{Mod}(M, [-,N]) \cong \underline{Mod}(M\triangleleft -, N)
	\]
\end{enumerate}
\end{theorem}
\begin{proof}
1. Composition of morphisms is a consequence of the universal property of measuring comodules. For modules $M_i$ over $A_i$, $i \in \{1, \ 2, \ 3\}$
\[
\xymatrixcolsep{2pc} 
\xymatrix{
Q(M_2,M_3)\otimes Q(M_1, M_2) \ar[r] &\Hom(M_2,M_3) \otimes \Hom(M_1,M_2) \ar[r] &\Hom(M_1, M_3)
}
\]
is a measuring map.

2. This is identical to 1. in its logic.

3. The great simplification of working within the categorical framework is that the parallels with the enriched categories of algebras and coalgebras are exact; the proof of this statement is formally identical to its parallel theorem \ref{thm:natequiven}. However, the added subtlety is that care must be taken to identify which algebras or coalgebras over which the modules or comodules are acting or coacting.

Let $X$ be a comodule over a coalgebra $H$. By Yoneda's lemma, it is sufficient to show that for a comodule $Y$ over $P(A,[H,B]) \cong P(A\triangleleft H, B) \cong P(H,P(A,B))$, there are natural isomorphisms of sets
\[
Comod(Y,Q(M\triangleleft X, N)) \cong Comod(Y, Q(M, [X,N]))
\]
\[
Comod(Y,Q(M, [X,N])) \cong Comod(Y, Q(X,Q(M,N)))
\]
The challenging aspect of the proof is to keep track of the underlying algebra or coalgebra homomorpshisms.

For the first, the sequence of equivalences is the following.
\begin{eqnarray*}
Comod(Y, Q(M\triangleleft X,N)) &  \cong ^a & Mod(M\triangleleft X, [Y,N]) \\
				&\cong^b  & Mod(M, [X, [Y,N]]) \\
				& \cong ^c & Mod(M, [X\otimes Y,N]) \\
				& \cong ^d & Mod( M, [Y\otimes X, N]) \\
				& \cong^e & Mod(M, [Y,[X,N]]) \\
				& \cong^f & Comod(Y, Q(M, [X,N])).			
\end{eqnarray*}
The equivalence $a$ is the defining property of the universal measuring comodule, and proposition \ref{prop:meascomod}. Equivalence $b$ is theorem \ref{thm:squaremodeq}, equivalences $ c,\  d, \ e$ are all familiar rearrangements. Equivalence $f$ is again theorem \ref{thm:squaremodeq}.

For the second the sequence of equivalences is the following.
\begin{eqnarray*}
Comod(Y, Q(X, Q(M,N)) & \cong^g & Comod(Y \otimes X, Q(M,N)) \\
				& \cong^h & Mod(M, [Y \otimes X,N])\\
				& \cong ^k& Mod (M, [Y,[X,N]]) \\
				& \cong^l & Comod( Y, Q(M, [X,N]).
\end{eqnarray*}
Equivalence $g$ is the fundamental property of the universal measuring comodule, proposition \ref{prop;meascomod}. Equivalence $h$ is propositions \ref{prop:natequivmod}.  Equivalence $k$ is a standard rearrangement, and $l$ is again proposition \ref{prop:natequivmod}.
\end{proof}
The enriched categories $\underline{Mod}$ and $\underline{Comod}$ provide global category of modules: \emph{no reference need be made to the algebra or coalgebra acting or coacting}. Either this enriched category structure, or the bicategory structure of $\widecheck{Mod}$ seems to be required to recover the more sensitive comparisons between modules. 

$Q(M,N)$ is a distinguished element (the terminal element) of $\widecheck{Mod}(M,N)$. There is no such distinguished element of $\widehat{Mod}(M,N)$; the following section describes the closest thing to it, the module analogue for the Sweedler product.

\subsection{$D(M,N) = (M\triangleleft N^\circ)$ and its properties}
As with algebras and the universal measuring coalgebras, the Sweedler product enables a pre-dual concept, which in turn provides methods to understand and compute the universal measuring comodule. We will need the analogue of the dual coalgebra $-^\circ$ for modules. 

\begin{proposition}
Let $N$ be an module over an algebra $B$.  Define the \emph{dual comodule} $N^\circ$ to the the subspace of the full dual $N^*$. 
\[
N^\circ = \{\nu \in N^*: \emph{ker} \nu \supset W, \  W \emph{ is a submodule},  \ dim N/W < \infty \}
\]
Then $N^\circ$ is a $B^\circ$ comodule.
\end{proposition}
\begin{proof}
Suppose $W$ is a submodule of $B$ such that $N/W$ is a finite dimensional $B$ module. The action $B\otimes N \to N$ then factors through $(B/Ann(N/W) )\otimes N/W \to (N/W)$. Thus any element $\nu$ in $n^\circ$ is an element of $W^\perp$ which is an $Ann(N/W)^\perp$ comodule.  But since $Ann(N/W)^\perp \subset B^\circ$ $W^\perp$ is equally a $B^\circ$ module. 
\end{proof}

\begin{definition}
The \emph{universal measuring module} $D(M,N)$ is defined to be the Sweedler product $M\triangleleft N^\circ$.
\end{definition} 

The categorical structure behind $D(M,N)$ is identical to that behind $F(A,B)$, enabling us to lift properties and proofs from algebras to modules.

\begin{proposition}\label{prop:propD} 
Properties of D(M,N).
\begin{enumerate}
	\item Let $B$ be a finite dimensional algebra. For any finite dimensional $B$ module $N$ and (arbitrary) $A$ module $M$, there is a module map $\tau(M,N)$ relative to $\eta: A \to F(A,B)\otimes B$ 
	\[
	\tau(M,N): M \to D(M,N)\otimes N
	\]
	such that for any extension $\rho: M \to W \otimes N$ relative to an extension $\sigma:A \to S \otimes B$ there is a module map $D(\rho):D(M,N) \to W$ relative to $F(\sigma)$ such that the diagram
	\[	
\xymatrixcolsep{5pc}
\xymatrix{
M\ar[rd]_\rho\ar[r]^{\tau(M,N)} & D(M,N)\otimes N \ar[d]^{D(\rho)\otimes 1_B}\\
&W\otimes N
}
\]
commutes. Moreover, $\tau(\_ ,N)$ is a natural transformation.
	\item For arbitrary $B$ and arbitrary $N$ the module map $D(\rho):D(M,N) \to W$ still exists. The assignment 
	\[
	D: \widehat{Mod(M,N)} \to Mod(D(M,N),-), \ \ D(W,\rho):= D(\rho): D(M,N) \to S
	\]
is a functor.
	\item If $N$ is finite dimensional $D(M,N)$ is an initial object in the category $\widehat{Mod}(M,N)$
	\item  The evaluation map $M \to [D(M,N),N]$ is a module map relative to the algebra homomorphism $A \to [P(A,B),B]$.
	\item $D(M,N)^\circ \cong Q(M,N).$
\end{enumerate}
\end{proposition}
The great advantage of the categorical approach is that the proofs of these statements are essentially the same as the corresponding statements for $F(A,B)$ in Theorem \ref{thm:F(A,B)}.
The module $D(M,N)$ is difficult to conceptualise, but we have some familiar examples.
\begin{corollary}
\begin{enumerate}
\item Considering $A$ as a left $A$ module, $D(A,A) = F(A,A)\otimes A^\circ$ as a left $F(A,A)$ module.
\item For $A = k$ and $M, \ N$ finite dimensional vector vector spaces, $D(M,N) = M\otimes N^*$.
\item For $A$ modules $M_1, \ M_2$ and $N$ a $B$ module, $D(M_1 \oplus M_2,N) = D(M_1,N) \oplus D(M_2,N)$.
\item For an $A$ module $M$ and $B$ modules $N_1, \ N_2$, $D(M, N_1\oplus N_2) = D(M,N_1) \oplus D(M,N_2)$
\end{enumerate}
\begin{proof}
\begin{enumerate}
	\item Observe that $a\otimes \alpha = \sum a\otimes \alpha_{(1)} \otimes 1_A\otimes \alpha_{(0)}$ for $a$ in $A$, $\alpha$ in $A^\circ$.  Identifying $A^\circ$ with $1_{F(A,A)} \otimes (1_A \otimes A^\circ)$, $D(A,A)$ is thus seen to be generated by $A^\circ$ as an $F(A,A)$ module.  To see that $D(A,A) = F(A,A)\otimes A^\circ$,
observe that he ideal $R$ defining the module is generated by the set $U$ of elements of the form $aa'\otimes \alpha - \sum(a\otimes \alpha_{(1)})\otimes (a' \otimes \alpha_{(0)})$. But
	\begin{align}
	\sum_{(\alpha)} (aa'\otimes \alpha_{(1) })\otimes (1_A\otimes \alpha_{(0)}) = & aa'\otimes \alpha \\
	= & \sum _{(\alpha)} (a\otimes \alpha_{(2)})\otimes(a'\otimes \alpha_{(1)})\otimes (1_A\otimes \alpha_{(0)})
	\end{align}
	But then, $\triangle ^2 = (\triangle \otimes 1)\circ \triangle$, for any expression $\triangle \alpha = \sum_\alpha \alpha_{(1)} \otimes \alpha_{(0)}$, and $U$ can be replaced by the set $U'$ of elements of the form
	\[
\sum_{(\alpha)} (aa'\otimes \alpha_{(1) })\otimes (1_A\otimes \alpha_{(0)}) -
	 \sum _{(\alpha)} (a\otimes \alpha_{(2)})\otimes(a'\otimes \alpha_{(1)})\otimes (1_A\otimes \alpha_{(0)})
	\]
	that is, elements
	\[
	\sum_{(\alpha)} (aa'\otimes \alpha_{(1) })-
	 \sum _{(\alpha)} (a\otimes \alpha_{(2)})\otimes(a'\otimes \alpha_{(1)})\otimes (1_A\otimes \alpha_{(0)})
	\]
	but each term 
	\[
	\sum_{(\alpha_{(0)})} (aa'\otimes \alpha_{(1) })-
	 \sum _{(\alpha_{(0)})} (a\otimes \alpha_{(2)})\otimes(a'\otimes \alpha_{(1)}))
	\]
	is 0 in $F(A,A)$.
	\item This is immediate since $F(k,k) = k$.
	\item $F(A,B)\otimes ((M_1\oplus M_2)\otimes N^\circ) = F(A,B)\otimes (M_1\otimes N^\circ) \oplus F(A,B) \otimes (M_2 \otimes N^\circ).$ The subspace $U$ generating the submodule $R$ defining $D(M_1 \oplus M_2, N)$ similarly splits as a direct sum $U_1 \otimes U_2 \subset F(A,B)\otimes (M_1\otimes N^\circ) \oplus F(A,B) \otimes (M_2 \otimes N^\circ)$
	\item This is similar to part 3.
\end{enumerate}
\end{proof}

\end{corollary}
\subsection{Coda: Internal symmetries?}

It is our hope that the material in this paper will excite the imagination of mathematicians working in fields beyond algebra and category theory.  One potential application that has motivated the research has been the possibility of finding a context for internal symmetries in physics: not a conjecture of what the "right" group should be (although the model would suggest various groups) but a reason why internal symmetries should be expected and a framework that would govern their operation.

The first observation is that for physical systems with a symmetry group $G$ acting, the representations of interest in quantum physics are complex representations in which $G$ is represented as unitary matrices. 

Let $A$ be the real algebra by two elements, $i$, complex multiplication, and $J$, complex conjugation. 

The theory of section 2 provides a canonical universal extension:
\[
\eta: A \to F(A,A)\otimes A
\]
The theory of section 4 provides, for \emph{any} representation $W$ of $A$ a canonical universal extension
\[
\tau: W \to D(W,W)\otimes W
\]
where $D(W,W)$ is an $F(A,A)$ module, so that $D(W,W) \otimes W$ is an $A$ module via the homomorphism $\eta$.

The theory thus provides a universal symmetry algebra, $F(A,A)$; its homomorphic images contain groups which could play the role of internal symmetry groups.  If $W$ is a (complex) module for the group $G$, $D(W,W)\otimes W $ is a module for $F(A,A)\otimes \mathbb{C}G$. If $Z$ is any module on which $F(A,A)$ acts, $F(A,A)\otimes \mathbb{C}G$ acts on $Z\otimes W$. If $H$ is then a group generated by a set of invertible elements in the image of $F(A,A) \in \End(Z)$, then $H\times G$ acts on $Z\otimes W$. In this way, groups $H$ would appear as inevitable extra symmetries given any complex representation of $G$.

Furthermore, $F(A,A)$ can be calculated explicitly, and its homomorphic images contain algebras and groups which have been used to model internal symmetries.  The procedure used to calculate $F(\mathbb{C},\mathbb{C})$ in section \ref{sec;basicexample} provides generators and relations for $F(A,A)$.

The canonical map $\eta: A \to F(A,A) \otimes A$ is a homomorphism thus writing $\eta(x) (resp\eta(J))$
\begin{align}
\eta(x) = & f_11 + f_xX  + f_JJ + f_{xJ}xJ \\
\eta(J) = & g_11 + g_xX  + g_JJ + g_{xJ}xJ 
\end{align}
we need 
\[
\eta(x)^2 = -1, \ \ \ \eta(J)^2 = 1, \ \ \ \{\eta(x),\eta(J)\} = 0.
\]
Direct computation of $\eta(x)^2$ gives
\begin{align}
-1 = & f_1^2 - f_x^2 + f_J^2 + f_{xJ}^2 \\
 0 = & \{f_1,f_x\} + [f_{xJ},f_J] = \{f_1,f_J\} + [f_{xJ},f_x] = \{f_1,f_{xJ\}}\} + [f_x,f_J]
\end{align}
and similarly
\begin{align}
1 = & g_1^2 - g_x^2 + g_J^2 + g_{xJ}^2 \\
 0 = & \{g_1,g_x\} + [g_{xJ},g_J] = \{g_1,g_J\} + [g_{xJ},g_x] = \{g_1,g_{xJ\}}\} + [g_x,g_J].
\end{align}
The identity $\{\eta(x), \eta(J)\}= 0$ gives
\begin{align}
0 = &\{f_1,g_1\} -\{f_x,g_x\} + \{f_j, g_J\} + \{f_{xJ}, g_{xJ}\} \\
=& \{f_1,g_x\} + \{g_1,f_x\} + [f_{xJ}, g_J] + [g_{xJ},f_J] \\
=& \{f_1,g_J\} + \{g_1,f_J\} + [f_{xJ}, g_x] + [g_{xJ},f_x] \\
=& \{f_1,g_{xJ}\} + \{g_1,f_{xJ}\} + [f_{x}, g_J] + [g_x,f_J].
\end{align}
Finite dimensional images of this algebra and $F(A,A)\otimes A$ contain many familiar algebras. For example, setting
\[
f_1^2 = f_x^2 = f_J^2 = f_{xJ} ^2 = 1/2, \ \ g_1=g_x = g_{xJ} = 0, \ \ g_j = 1
\]
and 
\[
\{f_1,f_x\} = \{f_1,f_j\} = \{f_1, f_{xJ}\} = [f_x,f_J] = [f_J,f_{xJ}] = [f_{xJ},f_x] =0
\]
contains a Clifford algebra with an inner product $\{1,-1,1,1\}$ generated by $\{f_1, f_xx, f_JJ,f_{xJ}xJ\}$.
Additionally requiring $f_J = f_{xJ}=0$ defines an algebra isomorphic to $F(\mathbb{C},\mathbb{C})$ whose complexification $F(\mathbb{C},\mathbb{C})\otimes \mathbb{C}$ has $M_2(\mathbb{C})$ as simple homomorphic images.

\newpage

\nocite{*}
\bibliography{boultonbatchelor} {}
\bibliographystyle{ieeetr}


\end{document}